%% file: second_version_arxiv.tex
	\documentclass[onefignum,onetabnum]{siamart171218}
	
	\input{ex_shared1}

        \usepackage{soul}
        \soulregister\cite{7}
        \soulregister\ref{7}
        \soulregister\eqref{7}

	\newcommand{\FNRm}[1]{}

	\newcommand{\g}{a}
	\ifpdf
	\hypersetup{
	  pdftitle={{PDE}-constrained OCP with uncertain parameters using SAGA},
	  pdfauthor={M.~Martin, and F.~Nobile}
	}
	\fi
	
\begin{document}
	
	\maketitle
	
	\begin{abstract}
	We consider an optimal control problem (OCP) for a partial
        differential equation (PDE) with random coefficients. The
        optimal control function is a deterministic, distributed
        forcing term that minimizes an expected quadratic regularized
        loss functional.
     	For the numerical approximation of this PDE-constrained OCP,
        we replace the expectation in the objective functional by a
        suitable quadrature formula and, eventually, discretize the
        PDE by a Galerkin method.
	To practically solve such approximate OCP, we propose an
        importance sampling version the SAGA algorithm
        \cite{DefazioBachLacosteJulien}, a type of Stochastic Gradient
        algorithm with a fixed-length memory term, which computes at
        each iteration the gradient of the loss functional in only one
        quadrature point, randomly chosen from a possibly non-uniform
        distribution.
	We provide a full error and complexity analysis of the proposed
	numerical scheme. In particular we compare the complexity of the
	generalized SAGA algorithm with importance sampling, with that of the Stochastic
	  Gradient (SG) and the Conjugate Gradient (CG) algorithms, applied to
	the same discretized OCP.
	We show that SAGA converges exponentially in the number of
	iterations as for a CG algorithm and has a similar asymptotic
	computational complexity, in terms of computational cost versus accuracy. 
        Moreover, it features good pre-asymptotic
	properties, as shown by our numerical experiments, which makes it
	appealing in a limited budget context. 
	\end{abstract}
	
	

        \section{Introduction}
        In this paper we consider a risk averse optimal control
        problem (OCP) for a linear coercive PDE with random parameters of the form
        \begin{equation}\label{eq:primal_intro}
          A_\omega y_\omega(u) = g_\omega + B_\omega u
        \end{equation}
        where $\omega\in\Gamma$ denotes a random elementary event, 
        $g_\omega$ is a random forcing term, 
        $A_\omega$ is a random coercive differential operator, $u$ is
        the \emph{deterministic} control function, whose action on the
        system is represented by the, possibly random, operator
        $B_\omega$, and $y_\omega(u)$ denotes the solution of the PDE
        corresponding to the control $u$ and random event $\omega$.
        The optimal control is the one that minimizes an expected
        tracking type quadratic functional
	\begin{equation}\label{eq:ocp_general}
	  u^* \in \argmin_{u\in U} J(u):= \E_\omega[f(u,\omega)], \quad \text{with } f(u,\omega) = \frac{1}{2}\|H_\omega y_\omega(u)-z_d\|_H^2+ \frac{\beta}{2}\|u\|^2_U.
	\end{equation}
        Here, $H_\omega$ is an observation operator and $z_d$ the
        (deterministic) target observed state. The objective
        functional $J(u)$ includes a possible penalization term for
        well posedness. In the setting considered in this work, the
        function $u \mapsto f(u,\omega)$ is strongly convex for any
        $\omega\in \Gamma$ with Lipschitz continuous gradient (Gateaux
        derivative). We present in Section \ref{sec2} of this work a
        specific example of a controlled diffusion equation with
        random diffusivity coefficient, which fits the general setting
        above.

        The PDE-constrained OCP
        \eqref{eq:primal_intro}-\eqref{eq:ocp_general} is
        infinite-dimensional and, as such, has to be suitably
        discretized to be solved numerically. In this work, we focus
        on an approximation of the expectation operator by a suitable
        (deterministic or randomized) quadrature formula $\E[X]
        \approx \sum_{j=1}^n \zeta_j X(\eta_j)$, with
        $X:\Gamma\to\mathbb{R}$ and integrable random variable,
        $\{\eta_j\}_j$ the quadrature knots and $\{\zeta_j\}_j$ the
        quadrature weights such that $\sum_{j=1}^n \zeta_j = 1$.  For
        example, one could use a Monte Carlo quadrature where $\eta_j$
        are drawn independently from the underlying probability
        measure and the quadrature weights are uniform $\zeta_j=1/n$.
        Alternatively, if the randomness in the system can be
        parameterized by a small number of independent random
        variables, one could use deterministic quadrature formulas
        such as a tensorized Gaussian one.  The use of tensorized full
        or sparse Gauss-type quadrature formulas to approximate
        statistics of the solution of a random PDE, also referred to
        as Stochastic Collocation method, has been widely explored in
        the literature in recent years. See
        e.g. \cite{NobileSto,nobile.tamellini.tempone:convergence,schillings.schwab:sparse,haji-ali.nobile.eal:MISC2,ernst.sprungk.tamellini:convergence,zech.dung.schwab:multilevel}
        and references therein.
        By introducing a quadrature formula, we are led to
        the approximate OCP
	\begin{equation}\label{eq:ocp_approx_general}
	  \widehat u^* \in \argmin_{u\in U} \widehat J(u):=\sum_{j=1}^n \zeta_j f(u,\eta_j), \quad \text{with } f(u,\eta_j) = \frac{1}{2}\|H_{\eta_j} y_{\eta_j}(u)-z_d\|_H^2+ \frac{\beta}{2}\|u\|^2_U
	\end{equation}
	and $y_{\eta_j}(u)$ solving equation \eqref{eq:primal_intro}
        for the particular realization $\eta_j$ of the randomness.
        
	To solve the approximate OCP \eqref{eq:ocp_approx_general},
        one could consider a gradient based method, such as steepest
        descent, hereafter called Full Gradient (FG), or Conjugate
        Gradient (CG), which converge, in our setting, exponentially
        fast in the number of iterations,
        i.e. $\|\widehat{u}_k-\widehat{u}^*\|_U \le C \rho^k$ for some
        $\rho\in (0,1)$ and a constant $C>0$ independent of the number
        $n$ of quadrature points, where $\widehat{u}_k$ denotes the
        $k$-th iterate of the method. For a given control $u$,
        evaluating the gradient of the objective functional $\nabla
        \widehat J(u)$ entails the computation of the $n$ solutions
        $\{y_{\eta_j}(u)\}_{j=1}^n$ of the underlying PDE as well as
        $n$ solutions of the corresponding adjoint equation.  The
        practical limitation of this approach is that, if $n$ is
        large, the cost of one single iteration may become already
        excessively high.
	  
        The approximate OCP \eqref{eq:ocp_approx_general} can be
        rewritten equivalently as
	\begin{equation}\label{eq:ML_objective}
	\widehat u^* \in \argmin_{u\in U} \widehat J(u):=\frac{1}{n}\sum_{j=1}^n g_j(u), \quad \text{with } g_j(u) = n\zeta_j f(u,\eta_j)
	\end{equation}
	which has the typical structure of an empirical risk
        minimization problem, common in statistical learning.  A
        popular technique in the machine learning community to solve
        this type of optimization problems is the \emph{Stochastic
          Gradient} (SG) method which reads, in its Robbins-Monro
        version \cite{robbins1951}
	\[
	  \widehat{u}_{k+1} = \widehat{u}_k - \tau_k  \nabla g_{i_k}(\widehat{u}_k)
	\]
	where the gradient of only one term in the sum is evaluated at
        each iteration (corresponding, in our setting, to one primal
        and one adjoint computation) for a randomly drawn index $i_k$,
        and the convergence is achieved by reducing the step size
        $\tau_k$ over the iterations. This makes the cost of each
        iteration affordable. The convergence of the SG method for a
        PDE-constrained optimal control problem with uncertain
        parameters involving a diffusion equation has been studied in
        the recent work \cite{MartinKrumscheidNobile} in the context
        of a Monte Carlo approximation of the expectation appearing in
        \eqref{eq:ocp_general}. In particular, we have shown that the
        root mean squared error $\E[\|\widehat u_k -\widehat
          u^*\|^2_U]^{\frac{1}{2}}$ of the SG method converges with
        order $1/\sqrt{k}$, which is the same order of the Monte Carlo
        ``quadrature'' error, and leads to an optimal strategy and a
        slightly better overall complexity than a FG or CG approach.
        If, however, a more accurate quadrature formula is used,
        providing a convergence order faster than $1/\sqrt{k}$, the SG
        method still converges with order $1/\sqrt{k}$ and would lead
        to a worse complexity than FG or CG.

	In recent years, variants of the SG method, such as the SAGA
        method \cite{DefazioBachLacosteJulien}, have been proposed for
        finite dimensional optimization problems of the form
        \eqref{eq:ML_objective}, which are able to recover an
        exponential convergence in the number of iterations. SAGA
        exploits a variance reduction technique and requires
        introducing a memory term which stores all previously computed
        gradients in the sum and overwrite a term if the corresponding
        index is re-drawn.  Specifically the SAGA algorithm for
        problem \eqref{eq:ML_objective} reads:
	\begin{equation}
	\widehat{u}_{k+1}=\widehat{u}_k-\tau \left(\nabla g_{i_k}(\widehat{u}_k)-\nabla g_{i_k}(\phi_{i_k}^k)+\frac{1}{n}\sum_{j=1}^{n}\nabla g_j(\phi_j^k)\right)
	\end{equation}
	where, at each iteration $k$, an index $i_k \in \{1, \dots,
        n\}$ is selected \emph{at random uniformly} on
        $\{1,\ldots,n\}$, and only the memory term at the sampled
        position $i_k$ is updated, so that
        $\phi_{j}^{k+1}=\widehat{u}_k$ if $j=i_k$ and
        $\phi_{j}^{k+1}=\phi_{j}^{k}$ otherwise.  The major
        improvement of SAGA with respect to SG is its exponential
        convergence $(1-\epsilon)^k$ provided that the objective
        functionals $g_i$ are all strongly-convex and with Lipschitz
        continuous gradients, with uniform bounds in $i$
        \cite{SchmidtLeRouxBach}. Such convergence is significantly
        faster than the $1/\sqrt{k}$ rate of SG.  However, the
        problem-dependent rate $\epsilon$ depends on the number $n$ of
        terms in the sum and behaves asymptotically as $1/8n$. As
        pointed out by the authors in \cite{SchmidtLeRouxBach}, this
        still guarantees an effective ($n$ independent) error
        reduction when going through a full sweep over the $n$ terms
        in the sum since $(1-1/8n)^n \leq \exp(-1/8)$. Thus, each pass
        through all the terms in the sum reduces the error by a
        constant multiplicative factor as in the FG or CG algorithm.

        In the context of our approximate OCP
        \eqref{eq:ocp_approx_general}, however, the objective
        functionals $g_j$ will not have uniform convexity and
        Lipschitz bounds in $i$, in general. Indeed, in our setting,
        this will be true for the functionals $f(u,\eta_i)$ but not
        for the functionals $g_i(u) = n\zeta_i f(u,\eta_i)$ due to the
        presence of the possibly non-uniform quadrature
        weights. Moreover, the OCP \eqref{eq:ocp_approx_general} is
        naturally infinite dimensional whenever the control is
        distributed. Therefore, the available results for SAGA can not
        be applied straightforwardly to our setting.

        The presence of non-uniform weights in 
        \eqref{eq:ocp_approx_general} suggests that the index $i_k$
        should possibly be drawn from a non-uniform distribution
        over $\{1,\ldots,n\}$.
        In this paper, we consider an extension of the SAGA algorithm
        that uses non-uniform sampling of the index $i_k$ at each
        iteration, from a given auxiliary distribution
        $\widetilde{\zeta}$ on $\{1,\ldots,n\}$, named SAGA with
        Importance Sampling (SAGA-IS), and we investigate the question
        of the optimal choice of the importance sampling distribution
        $\widetilde{\zeta}$.  We should point out that non-uniform
        sampling within SAGA has already been proposed in
        \cite{SchmidtDefazioBabanezhad}, however, with the goal of
        compensating possible dissimilarities in the Lipschitz
        constants of the functionals $g_i$, whereas these are still
        assumed to be uniformly strongly convex. Our analysis and
        choice of the importance sampling distribution differs
        therefore from \cite{SchmidtDefazioBabanezhad}, as it serves a
        different purpose.

        Following similar steps as in
        \cite{DefazioBachLacosteJulien,SchmidtDefazioBabanezhad}, we
        present a full theoretical convergence analysis of the SAGA-IS
        method for the infinite dimensional OCP
        \eqref{eq:ocp_approx_general} in the case of positive
        quadrature weights. In particular we show that, asymptotically
        in $n$, the optimal importance sampling measure for the
        indices $i_k$ is the \emph{uniform} measure (even when the
        weight are non uniform) and that the error decays, in a mean
        squared sense, as $\|\widehat{u}_k-\widehat{u}^*\|\le C
        (1-\epsilon)^k$ with $\epsilon\sim n^{-1}$ and $C$ independent
        of $n$ (similarly to the results in \cite{DefazioBachLacosteJulien}).
	
	We also present a complexity analysis, in terms of
        computational cost versus accuracy, of the SAGA-IS method to
        solve the \emph{original} OCP \eqref{eq:ocp_general}, which
        accounts for both the quadrature error as well as the error in
        solving the primal and adjoint PDEs approximately by a
        Galerkin method. Here, we assume that the quadrature error
        decays sub-exponentially fast in the number of quadrature
        points, an assumption that is verified by the controlled
        diffusion problem presented in Section \ref{sec2} when using a
        tensorized Gauss-Legendre quadrature formula. The complexity
        of SAGA is then compared to the complexity of CG as well as
        SG. Our theoretical results show that the SAGA method has the
        same asymptotic complexity as the CG method and outperforms
        SG.
	As shown by our numerical experiments, the interest in using
        SAGA versus CG is in the pre-asymptotic regime, as SAGA
        delivers acceptable solutions, from a practical point of view,
        well before performing $n$ iterations, i.e. with far less that
        $2n$ PDE solves (we recall that one single CG iteration
        entails already $2n$ PDE solves).  In a context of limited
        budget, SAGA represents therefore a very appealing option.

        The outline of the paper is as follows. In Section \ref{sec15}
        we introduce the general PDE-constrained OCP considered in
        this work as well as its discretized version where expectation
        in the objective functional is approximated by a quadrature
        formula and the PDE is possibly approximated by a Galerkin
        method. In Section \ref{sec2} we give a specific example of a
        controlled diffusion equation with random coefficients, which
        fits the general framework. Then, in Section \ref{sec4} we
        introduce the Stochastic Gradient and SAGA methods with
        importance sampling and study their convergence and
        complexity. In Section \ref{sec5} we present some numerical
        results that confirm the theoretical findings. Finally, in
        Section \ref{conclusions} we draw some conclusions.

	\section{General PDE-constrained quadratic OCP}
	\label{sec15} 
	\subsection*{The state problem}
	We consider a physical system whose state $y_\omega$ is well described by the solution of a linear, coercive, PDE, which contains some random parameters $\omega\in\Gamma$ describing, for instance, randomness in the forcing terms,  or uncertainties in some parameters of the PDE. We write the PDE in abstract operator form as 
	\begin{equation} \label{eq:PDErdm}
	A_\omega y_\omega=g_\omega +B_\omega u, \quad u \in U, \quad y_\omega \in Y, \quad \omega \in \Gamma.
	\end{equation}
	where $g_\omega$ is a random forcing term and $u$ is a, possibly
	distributed, control term that can be used to modify the behavior of
	the system and drive it to a target state.  We admit that the action
	of the control on the system, expressed here by the operator
	$B_\omega$ could also be affected by randomness or uncertainty.
	
	In \eqref{eq:PDErdm}, we denote by $(\Gamma,\mathcal{F}, \mathbb{P})$ the underlying (complete) probability space, by $U$ the set of admissible control functions, assumed here to be a Hilbert space, and by $Y$ the space of solutions to the PDE, assumed here to be a Banach space. We also assume that equation \eqref{eq:PDErdm} holds in $Y'$, the dual space of $Y$ and that $g_\omega\in Y'$, $A_{\omega} \in\mathcal{L}(Y, Y')$, $B_\omega \in \mathcal{L}(U, Y')$ for $\mathbb{P}$-almost every $\omega\in\Gamma$.
	
	The solution of \eqref{eq:PDErdm} for a given control $u$ and a given realization $\omega\in\Gamma$ will be equivalently denoted by $y_\omega (u)$, $y_\omega$ or simply $y(u)$ in what follows, and is given by the formula $y_{\omega}(u)=A_\omega^{-1}(g_\omega+B_\omega u)$, provided the operator $A_\omega$ is invertible.

	\subsection*{The Optimal Control Problem}
	We focus on some part of the state, or some observable quantity of interest $Q_\omega y_\omega \in H$, with $H$ a suitable Hilbert space and $Q_\omega\in\mathcal{L}(Y,H)$ the observation operator, which may as well have some random components.
	The control $u\in U$ can be used to drive the quantity of interest $Q_\omega y_\omega$ to a target value $z_d$. For a given control $u\in U$ and a given realization $\omega\in\Gamma$, we define the quadratic misfit function
	\[
	\tilde f(u,\omega) := \frac{1}{2}\left\|Q _\omega y_\omega-z_{d}\right\|_{H}^{2}.
	\]
	An optimal control should aim at minimizing such function, which is, however, random as it depends on the realization $\omega$. We assume in this work that the randomness is not observable at the moment of designing the control, and look therefore for an optimal control that is \emph{robust} with respect to the randomness. In particular, we consider the optimal control that minimizes the expected quadratic loss plus an $L^2$-regularization term for well posedness. This leads to the following OCP:
	\begin{equation} \label{eq:OCP0}
	\gathered
	  \uu \in \argmin _{ u \in  U} J(u):=\E \left[f(u, \omega)\right]\\
	\text{with } f(u, \omega)=\frac{1}{2} \left\|Q_\omega  y_\omega(u)-z_{d}\right\|_{H}^{2}+\frac{\beta}{2} \|u\|_{U}^{2} \text{ and } y_\omega(u)=A_\omega^{-1}(g_\omega+B_\omega u).
	\endgathered
	\end{equation}
	It is worth introducing also a weak formulation of the OCP \eqref{eq:OCP0}. 
	Let $b_{\omega}:Y\times Y\to \mathbb{R}$  denote the bi-linear form associated to the operator  $A_\omega$, namely $b_{\omega}(y,v) = \left(A_\omega y,v\right)_{Y,Y'}, \; \forall y,v\in Y$. Then, the weak formulation of the linear PDE \eqref{eq:PDErdm} reads
	\begin{equation} \label{eq:weak_primal}
	\mathrm{find} \hspace{1mm} y_{\omega} \in Y \hspace{1mm} s.t. \hspace{1mm} b_{\omega}(y_{\omega},v)=\left( g_{\omega}+B_{\omega}u,v \right)_{Y,Y'}  \quad \forall v \in Y, \quad  \text{for $\mathbb{P}$-a.e. } \omega \in \Gamma,
	\end{equation} 
	and we can rewrite the OCP \eqref{eq:OCP0} equivalently as:
	\begin{equation} \label{eq:weak_OCP}
	\left\{ 
	\arraycolsep=1.0pt\def\arraystretch{1}
	\begin{array}{l}
	\min_{u \in  U}J(u), \quad J(u)=\E\left[ f(u, \omega) \right] \\
	\mathrm{s.t.} \quad y_{\omega}(u) \in Y \quad \mathrm{solves}\\
	b_{\omega}(y_{\omega}(u),v)=\left(g_{\omega}+B_{\omega}u,v \right)_{Y,Y'},  \quad \forall v \in Y, \quad  \text{for $\mathbb{P}$- a.e. } \omega \in \Gamma.
	\end{array}
	\right.
	\end{equation}

	\subsection*{{Assumptions for well posedness}}
	We present here a set of assumptions on the OCP \eqref{eq:OCP0} that guarantee its well posedness. Such assumptions are an easy generalization of those in \cite{Hinze2009} to the stochastic setting and also generalize the setting in \cite[Sections 2 and 3]{MartinKrumscheidNobile}.
	
	\begin{assumption}\label{ass:general}
	~
	\begin{enumerate}
	\item $\beta> 0$;
	\item $A_\omega \in L_{\Prob}^{\infty}\left(\Gamma;  \mathcal{L}(Y,Y')\right)$, i.e. $\exists a_{\text{max}}>0$ such that for $\mathbb{P}\text{-a.e. } \omega \in \Gamma$, $A_\omega \in \mathcal{L}(Y, Y')$ and  $\| A_\omega \|_{\mathcal{L}(Y,Y')}\leq a_{\text{max}}$;
	\item For $\mathbb{P}\text{-a.e. } \omega \in \Gamma$, $A_\omega$ is invertible with uniformly bounded inverse, and $\exists a_{\text{min}}>0$ s.t. ${a_{\min}}\|y\|_Y^2 \leq ( A_\omega y,y )_{Y,Y'}$;
	\item $g \in L_{\Prob}^2(\Gamma; Y')$, i.e. $\int_{\Gamma} \| g_\omega \|^2_{Y'} \text{d}\Prob(\omega)<+\infty$;
	\item $B \in L_{\Prob}^{\infty}\left(\Gamma; \mathcal{L}(U, Y')\right)$, i.e. there exists $M_1 \in \R$ s.t. $\| B_\omega \|_{ \mathcal{L}(U, Y')} \leq M_1$, for $\mathbb{P}\text{-a.e. } \omega \in \Gamma$;
	\item $Q  \in L_{\Prob}^{\infty}\left(\Gamma; \mathcal{L}(Y, H)\right)$, i.e. there exists $M_2 \in \R$ s.t.  $ \| Q_\omega \|_{ \mathcal{L}(Y, H)} \leq M_2$; for $\mathbb{P}\text{-a.e. } \omega \in \Gamma$;
	\end{enumerate}
	\end{assumption}

	\subsection*{Gradient computation and adjoint problem}
	
	In what follows, we denote by $\nabla J(u)$ the functional representation of the Gateaux derivative of $J$ in the space $U$, defined as 
	\[
	\int_D \nabla J(u) \delta u \hspace{1mm}\mathrm{d}x = \lim_{\epsilon \rightarrow 0} \frac{J(u+\epsilon \delta u)-J(u)}{\epsilon},  \quad \forall \, \delta u \in U.
	\]
	Moreover, the adjoint of a linear operator $C:W\to Z$ between linear
	spaces, will be defined via duality paring $(\cdot,\cdot)_{W,W'}$
	(resp $(\cdot,\cdot)_{Z,Z'}$) whenever $W$ (resp. $Z$) is a reflexive Banach
	space and via the inner product $\langle\cdot,\cdot\rangle_W$
	(resp. $\langle\cdot,\cdot\rangle_Z$) whenever $W$ (resp. $Z$) is a
	Hilbert space. So, for instance, the adjoint of $A_\omega$ satisfies
	\[
	A^*_\omega\in\mathcal{L}(Y,Y'), \qquad (A^*_\omega v,y)_{Y,Y'}=(A_\omega y,v)_{Y,Y'}, \;\; \forall y,v\in Y
	\]
	whereas the adjoints of the operators $B_\omega$ and $Q_\omega$ satisfy
	\begin{align*}
	B^*_\omega\in\mathcal{L}(Y,U), \qquad \langle B^*_\omega v,u\rangle_{U} = (B_\omega u,v)_{Y,Y'}, \;\; \forall u\in U, \; v\in Y, \\
	Q^*_\omega\in\mathcal{L}(H,Y), \qquad (Q^*_\omega z,y)_{Y,Y'}= \langle Q_\omega y,z\rangle_{H}, \;\; \forall y\in Y, \; z\in H.
	\end{align*}
	With these definitions, we have the following characterization of $\nabla J(u)$ and $\nabla f(u,\omega)$.
	\begin{lemma}\label{lem:gradient} Let Assumption \ref{ass:general} hold. For the OCP \eqref{eq:OCP0} it holds for any $u\in U$ and $\mathbb{P}\text{-a.e. } \omega\in \Gamma$
	  \begin{equation}
	    \nabla f(u,\omega) = \beta u + B^*_\omega p_\omega, \qquad \nabla J(u) = \E\left[ \nabla f(u,\cdot)\right]=\beta u +\E[B^*_\omega p_\omega]
	  \end{equation}
	  where $p_\omega$ satisfies the adjoint equation
	  \begin{equation} \label{eq:adjoint_general}
	  A_{\omega}^*p_{\omega}=Q_\omega ^{*}(Q_\omega y_\omega-z_d),
	  \end{equation}
	  which, in weak formulation reads
	  \begin{equation} \label{eq:adjoint_weak}
	\mathrm{find} \hspace{1mm} p_{\omega} \in Y \hspace{1mm} s.t. \hspace{1mm}b_{\omega}(v,p_{\omega})=\langle Q_\omega  v,Q_\omega  y_{\omega}-z_d \rangle_ H  \quad \forall v \in Y, \quad  \text{for $\mathbb{P}$- a.e. } \omega \in \Gamma.
	\end{equation}
	\end{lemma}
	\begin{proof}
	Denoting $\gtomega = Q_\omega  A_\omega^{-1}g_\omega-z_d$, we can write the quadratic functional $f(u,\omega)$ for $\mathbb{P}\text{-a.e. } \omega\in \Gamma$ as
	\begin{align*}
	f(u,\omega)  = &  \frac{1}{2} \| Q_\omega  A_\omega^{-1}(g_\omega+B_\omega u)-z_d \|_H^2 + \frac{\beta}{2} \| u \|_ U^2\\
	= & \frac{1}{2} \| Q_\omega  A_\omega^{-1}B_\omega u + \gtomega \|_H^2 +\frac{\beta}{2} \| u \|_ U^2\\
	 = &\frac{1}{2} \| Q_\omega  A_\omega^{-1}B_\omega u \|_H^2 + \langle Q_\omega  A_\omega^{-1}B_\omega u,\gtomega  \rangle_H 
	+\frac{1}{2} \| \gtomega \|_H^2 +\frac{\beta}{2} \| u \|_ U^2\\
	=& \frac{1}{2} \langle B_\omega^* A_\omega^{-*}Q_\omega ^* Q_\omega  A_\omega^{-1} B_\omega u ,u \rangle_{U} 
	+\langle B^*_\omega A_\omega^{-*} Q^*_\omega  \gtomega, u  \rangle_U \\
	&+\frac{1}{2} \| \gtomega \|_H^2 
	+\frac{\beta}{2} \| u \|_U^2
	\end{align*}
	and its Gateaux derivative at $u$ in the direction $\delta u\in U$ is given by
	\begin{equation}
	\langle \nabla f(u,\omega), \delta u\rangle_{U}= \langle B_\omega^* A_\omega^{-*}Q_\omega ^* Q_\omega  A_\omega^{-1} B_\omega u ,\delta u \rangle_{U} + \langle  B_\omega^* A_\omega^{-*}Q_\omega ^* \gtomega, \delta  u  \rangle_{U} +\beta  \langle  u, \delta  u  \rangle_{U}
	\end{equation}
	which leads to 
	\begin{align*}
	\nabla f(u,\omega)=&\beta u+ B_\omega^* A_\omega^{-*}Q_\omega ^*(Q_\omega  A_\omega^{-1} (B_\omega u+g_\omega)-z_d) \\
	=&\beta u+ B_\omega^* A_\omega^{-*}Q_\omega ^*(Q_\omega  y_\omega  -z_d ) \\
	=&\beta u+ B_\omega^* p_\omega
	\end{align*}
	and
	\[
	  \nabla J(u) = \E[\nabla f(u,\omega)] = \beta u+ \E[B_\omega^* p_\omega].
	\]
	
	\,
	\end{proof}
	Moreover, $\nabla f(u,\omega)$ satisfies the following uniform Lipschitz and strong convexity properties.
	\begin{lemma} Under Assumption \ref{ass:general}, it holds for any $u_1,u_2\in U$ and $\mathbb{P}\text{-a.e. } \omega\in\Gamma$
	  \begin{align}
	    &\text{(Strong convexity)} &&  \frac{l}{2} \| u_1-u_2 \|_U^2 \leq \langle \nabla f(u_1,\omega)-\nabla f(u_2,\omega) , u_1-u_2 \rangle_U,
	  \label{lemma:SC} \\
	    &\text{(Lipschitz property)} && \|\nabla f(u_1,\omega)-\nabla f(u_2,\omega)\|_U \leq L \| u_1-u_2 \|_U, \label{lemma:Lip} 
	  \end{align}
	  with $l=2 \beta$ and $L=\beta +\left(\frac{M_1M_2}{a_{\min}}\right)^2$.
	\end{lemma}
	\begin{proof}
	  To prove the strong convexity property we proceed as follows, using Assumption \ref{ass:general}
	  \begin{align*}
	    &\langle \nabla f(u_1,\omega)-\nabla f(u_2,\omega) , u_1-u_2 \rangle_U \\
	    & \hspace{2cm} = \langle \beta(u_1-u_2) + B^*_\omega A^{-*}_\omega Q^*_\omega Q_\omega A_\omega^{-1} B_\omega (u_1-u_2) , u_1-u_2 \rangle_U \\
	    & \hspace{2cm} = \beta\|u_1-u_2\|^2_U + \|Q_\omega A_\omega^{-1} B_\omega (u_1-u_2)\|^2_H \ge \beta\|u_1-u_2\|^2_U.
	  \end{align*}
	  Similarly, to prove the Lipschitz property we have
	  \begin{align*}
	\| \nabla f(u_1,\omega)-\nabla f(u_2,\omega)\|_U & \leq \|\beta (u_1 - u_2)+ B^*_\omega A^{-*}_\omega Q^*_\omega Q_\omega A_\omega^{-1} B_\omega (u_1-u_2)\|_U \\
	& \leq \left(\beta +\|B_\omega\|_{\mathcal{L}(U,Y')}^2  \| A_{\omega}\|_{\mathcal{L}(Y,Y')}^{-2}  \|Q_\omega \|_{\mathcal{L}(Y,H)}^2\right)   \|u_1-u_2\|_U.
	\end{align*}
	\end{proof}
	It follows that the OCP \eqref{eq:OCP0} is a quadratic (strongly) convex optimization problem hence admitting a unique optimal control $u^*\in U$.

	\subsection*{Space discretization}
	\label{gen:FE}
	We introduce now a Galerkin approximation of the state equation \eqref{eq:weak_primal} and of the OCP \eqref{eq:weak_OCP}.  Let $\{Y^h\subset Y\}_{h>0}$ and $\{U^h\subset U\}_{h>0}$ be sequences of finite dimensional spaces, indexed by a discretization parameter $h>0$, e.g. the mesh size for a finite element discretization, which are asymptotically dense in the respective spaces $Y$ and $U$ as $h\to 0$, namely
	\[
	  \lim_{h\to 0}\inf_{y^h\in Y^h} \|y-y^h\|_Y = 0, \; \forall y\in Y; \qquad 
	  \lim_{h\to 0}\inf_{u^h\in U^h} \|u-u^h\|_U = 0, \; \forall u\in U.
	\]
	We consider the discretized state equation for a given approximate control $u^h\in U^h$
	\begin{equation} \label{eq:PDErdmh_weak}
	\text{find $y_\omega^h\in Y^h$ s.t.} \;\; b_\omega(y_\omega^h(u^h),v^h)=(g_{\omega}+B_{\omega} u^h,v^h), \quad v^h \in Y^h, \quad \text{for $\mathbb{P}$-a.e. } \omega \in \Gamma,
	\end{equation}
	and the discretized (in space) OCP:
	\begin{equation} \label{eq:FE_OCP}
	  u^{h*} \in \argmin _{ u^h \in  U^h} J^h(u^h):= \E \left[f^h(u^h, \omega)\right]  
	\end{equation}
	with $f^h(u^h, \omega)=\frac{1}{2} \left\|Q_\omega  y^h_\omega(u^h)-z_{d}\right\|_{H}^{2}+\frac{\beta}{2} \|u^h\|_{U}^{2}$
	where $y_\omega^h(u^h)$ satisfies \eqref{eq:PDErdmh_weak}.
	By introducing the discrete operator  $A^h_{\omega} \in\mathcal{L}(Y^h, (Y^{h})')$, defined through the bilinear form $b_\omega$ restricted to the subspace $Y^h$, i.e. $\left( A^h_{\omega} y^h,v^h \right)_{Y,Y'} = b_{\omega}(y^h,v^h)$ for $y^h,v^h \in Y^h$, the discretized functional $f^h(u^h,\omega)$ can be written in operator form as
	\[
	f^h(u^h,\omega) = \frac{1}{2} \left\|Q_\omega  (A_\omega^h)^{-1}(g_\omega+B_\omega u^h)-z_{d}\right\|_{H}^{2}+\frac{\beta}{2} \|u^h\|_{U}^{2}.
	\]
	Notice that $A_\omega^h$ is almost surely coercive and invertible thanks to Assumption \ref{ass:general}.3, with the same uniform bound $\|(A_\omega^h)^{-1}\|_{\mathcal{L}(Y^h,(Y^h)')}\le a_{min}^{-1}$ on its inverse as for the continuous operator $A_\omega$. By proceeding as in Lemma \ref{lem:gradient} the gradient of $f^h$ in $U^h$ (i.e. the functional representation of the Gateaux derivative of $f^h$ in $U^h$) is given by
	\[
	\nabla f^h(u^h,\omega) = \beta u^h + \Pi^h_U B_\omega^* p_\omega^h
	\]
	where $p_\omega^h$ solves the discrete adjoint equation
	\begin{equation} \label{eq:adjointh_weak}
	b_\omega(v^h,p_\omega^h)=\langle Q_\omega  v^h,Q_\omega  y_{\omega}^h-z_d \rangle_ H  \quad \forall v^h \in Y^h, \quad  \text{for $\mathbb{P}$- a.e. } \omega \in \Gamma,
	\end{equation}
	and $\Pi^h_U\in\mathcal{L}(U,U^h)$ denotes the $U$-orthogonal projector on $U^h$
	\[
	\langle\Pi^h_U u, w^h\rangle_U = \langle u, w^h\rangle_U, \qquad \forall w^h\in U.
	\]
	Also, it is easy to check that the discretized functional $f^h(u^h,\omega)$ satisfies the same \emph{strong convexity} and \emph{Lipschitz} properties as its continuous counterpart, with exactly the same constants $l$ and $L$, respectively. This comes from the fact that the bounds on $\|A_\omega^{-1}\|$ and $\|(A_\omega^h)^{-1}\|$ are the same and $\|\Pi^h_U\|=1$. Hence, also the discretized OCP \eqref{eq:FE_OCP} admits a unique optimal control $u^{h*}\in U^h$.
	
	For the complexity analysis that will be carried out in Section \ref{sec:complexity}, we further assume a certain algebraic decay rate of the error $\|u^*-u^{h*}\|_U$ as a function of the discretization parameter $h$.
	\begin{assumption}\label{ass:FErate}
	There exist $p \in \N_{>0}$ and a constant $A_1$ independent of $h$ such that
	\begin{equation}
	\left\| \uu -\uuh \right\| \leq A_1 h^{p+1}, \qquad \forall h>0.
	\end{equation}
	\end{assumption}
	Such assumption is standard in the context of finite element approximation of elliptic problems and will be verified in Section \ref{sec2} for a specific instance of the OCP \eqref{eq:OCP0}.

	\subsection*{Discretization in probability}
	To discretize the OCP \eqref{eq:OCP0} in probability, we replace the exact expectation operator $\E[\cdot]$ by a discretized version $\widehat{E}[\cdot]$ of the form
	\begin{equation}\label{eq:quadrature}
	\widehat{E}[X]=\sum_{j=1}^n \zeta_j X(\eta_j),
	\end{equation}
	where $n$ is the total number of collocation points used, $\{\eta_j\}_{j=1}^n\subset\Gamma$ are the collocation points and $\{\zeta_j\}_{j=1}^n\subset\mathbb{R}$ the associated quadrature weights. Therefore, the discretized (in probability) OCP reads:
	\begin{equation}
	\label{eq:semi_discret_gen}
	  \uhu \in \argmin_{\uhat \in  U}\widehat{J}(\uhat)=\sum_{j=1}^n \zeta_j f_j(\uhat), \quad f_j(\uhat)=\frac{1}{2}\|Q_{\eta_j} y_{\eta_j}(\uhat) -z_d\|^2_H + \frac{\beta}{2}\|\uhat\|^2_U
	\end{equation}
	where each $y_{\eta_j}(\uhat)\in Y$, $j=1,\ldots,n$ satisfies the state equation 
	\[
	  b_{\eta_j}(y_{\eta_j}(\uhat),v)= (g_{\eta_j}+B_{\eta_j} \uhat,v)_{Y,Y'}  \quad \forall v \in Y.
	\]
	It is immediate to see that if all quadrature weights are positive, then the OCP \eqref{eq:semi_discret_gen} is (strongly) convex and admits a unique minimizer $\uhat\in U$. This is not necessarily true is the quadrature weight can have alternating sign. Therefore, in what follows we will often make the assumption of positive weights.

	An approximation of the type \eqref{eq:quadrature} can be achieved,
	for instance by a Monte Carlo method, where the quadrature knots
	$\eta_j$ are drawn independently from the probability measure $\Prob$
	and the quadrature weights are uniform $\zeta_j=1/n$ and
	positive. Alternatively, in many applications, the randomness can be
	expressed in terms of a finite number of independent random variables
	$\xi=(\xi_1,\dots, \xi_M)$, e.g. uniformly distributed, in which case
	a deterministic quadrature formula such as a tensorized Gauss-Legendre
	one, or a Quasi Monte Carlo formula could be used, both having
	positive (but not necessarily uniform) weights.
	
	For the complexity analysis that will be carried out in Section
	\ref{sec:complexity}, we make the following assumption on the decay of the 
	error $\|u^*-\uhat^{*}\|_U$ as a function of the number of quadrature knots.
	\begin{assumption}[convergence in probability]
	\label{ass:conv_prob}
	There exists a positive decreasing function $\theta: \N \rightarrow \R_+$ such that
	\[
	\| \uu-\uhu \| \leq \theta(n), \;\; \forall n\in \mathbb{N}_{>0} \quad \text{and} \quad \lim_{n \to +\infty}\theta(n) = 0.
	\]
	\end{assumption}

        \subsection*{{The fully discrete OCP}}
	Clearly, the discretizations in space and in probability described above can be combined to obtain a fully discrete OCP:
	\begin{equation}\label{eq:full_discret_gen}
	\gathered
	  \uhhu \in \argmin_{\widehat{u}^h \in  \widehat{U}^{h}}\widehat{J}^h(\widehat{u}^h)=\sum_{j=1}^n \zeta_j f_j^h(\widehat{u}^h), \quad f_j^h(\widehat{u}^h)=\frac{1}{2}[\| Q_{\eta_j}y^h_{\eta_j}(\widehat{u}^h)-{z_d} \|_H^2]+\frac{\beta}{2}\| \widehat{u}^h  \|_U^2
	\\
	\text{and $y_{\eta_j}^h \in Y^h$:}
	\quad b_{\eta_j}(y^h_{\eta_j}(\widehat{u}^h),v^h)=(B_{\eta_j}\widehat{u}^h+g_{\eta_j},v^h)_{Y,Y'}, \quad \forall v^h \in Y^h, \; j=1,\ldots,n.
	\endgathered
	\end{equation}
	 If $\uhhu$ denotes the solution of the OCP
         \eqref{eq:full_discret_gen} using $n$ knots in the quadrature
         formula $\widehat{E}[\cdot ]$, {under Assumptions \ref{ass:general}-\ref{ass:conv_prob},} the total error will satisfy:
	\begin{equation}
	\| \uu-\uhhu \| \leq A_1 h^{p+1}+\theta(n).
	\end{equation}
	
	We particularize the last assumptions and convergence results for a specific elliptic problem in the next section.

	\section{Diffusion equation with random coefficients}\label{sec2}
	In this section we provide a specific example of a controlled
	diffusion equation with random coefficients, which fits the general
	framework introduced in the previous section. Let $D \subset \R^d$
	denote the physical domain (open bounded subset of $\R^d$); for a
	given (deterministic) control $u:D\to \R$, the state equation reads:
	find $y: D \times \Gamma \rightarrow \R$ such that
	\begin{equation}
	\left\{
	\begin{array}{rcll}
	-\dive(a(x,\omega)\nabla y(x,\omega)) & = & g(x)+u(x), \hspace{10mm} & x \in D, \quad \text{for $\mathbb{P}$-a.e. } \omega \in \Gamma, \\
	y(x,\omega) & = & 0,  \hspace{10mm} & x \in \partial D, \quad \text{for $\mathbb{P}$-a.e. } \omega \in \Gamma, 
	\end{array}
	\right.
	\label{eq:primal}
	\end{equation}
	where $g:D\to\R$ is a deterministic forcing term and the diffusion coefficient $a$ is a random field uniformly bounded and positive, i.e. such that
	\begin{equation}\label{eq:uniform_a}
	\exists \,  a_{\mathrm{min}}, a_{\mathrm{max}} \in \R_+ \hspace{4mm} \mathrm{such} \hspace{1mm} \mathrm{that} \hspace{4mm} 0<a_{\mathrm{min}}\leq a(x,\omega) \leq  a_{\mathrm{max}} \hspace{4mm} \mathrm{a.e.} \hspace{1mm} \mathrm{in} \hspace{1mm} D \times \Gamma.
	\end{equation}
	We set $U= L^2(D)$ as the space of admissible control functions and $Y=H_0^1(D)$, endowed with the norm $\|v\|_Y=\|\nabla v\|_{L^2(D)}$, as the space of solutions of \eqref{eq:primal}, assuming $g\in Y'$. Given a target function $z_d\in U$ and $\beta>0$, our goal is to find the optimal control $u^*\in U$ that satisfies
	\begin{equation}\label{eq:ocp}
	  \uu \in \argmin_{u \in U} J(u)=\E[f(u,\omega)], \quad f(u, \omega)=\frac{1}{2}\|y_{\omega}(u) - z_d\|_U^2 +\frac{\beta}{2}\|u\|_U^2
	\end{equation}
	where $y_\omega(u)$ satisfies, in a weak sense, equation
	\eqref{eq:primal} for the given control $u$.  For this problem we have
	therefore, $H=U$, $B_\omega=Q_\omega=Id$, the identity operator in $U$,
	and $A_\omega$ the operator associated to the PDE \eqref{eq:primal}, which is uniformly coercive and bounded (and with uniformly bounded inverse) thanks to the condition \eqref{eq:uniform_a}. It is not difficult to verify that Assumptions \ref{ass:general}.1-6 are satisfied for this problem, hence the OCP \eqref{eq:ocp} admits a unique solution $u^*\in U$. 
	The gradient of the functionals $J(u)$ and $f(u,\omega)$, for a given control $u\in U$, can be computed as  
	\begin{equation}
	\label{gradi}
	\nabla f(u,\omega)=\beta u + p_{\omega}(u), \qquad 
	\nabla J(u)=\beta u +\E[p_{\omega}(u)],
	\end{equation}
	where $p_{\omega}(u)=p$ is the solution of the adjoint problem
	\begin{equation}
	\left\{
	\begin{array}{rcll}
	-\dive(a(\cdot,\omega)\nabla p(\cdot,\omega)) & = & y(\cdot,\omega)-z_d \hspace{10mm} & \mathrm{in}\hspace{1mm} D, \quad \text{for $\mathbb{P}$-a.e. } \omega \in \Gamma,\\
	p(\cdot,\omega) & = & 0  \hspace{10mm} &  \mathrm{on}\hspace{1mm} \partial D, \quad \text{for $\mathbb{P}$-a.e. } \omega \in \Gamma. 
	\end{array}
	\right.
	\label{eq:dual}
	\end{equation}
	and the functional $f(u,\omega)$ satisfies the strong convexity and Lipschitz properties in Lemma \ref{lemma:SC} with constants $l=2 \beta$ and $L=\beta+\frac{C_p^4}{a_{min}^2}$, where $C_p$ is the Poincar\'e constant, i.e. $C_p=\sup_{v \in Y/\{0\}} \frac{\| v \|_{L^2(D)}}{\| \nabla v \|_{L^2(D)}}$.  (See also \cite[Sections 2 and 3]{MartinKrumscheidNobile} for more details.)
	
	We further assume, in what follows, that the randomness in
        \eqref{eq:primal} can be parametrized in terms of a
        \emph{finite number} of independent and identically
        distributed uniform random variables $\xi_i$, $i=1,\ldots,M$
        on $[-1,1]$.  Hence, in this case, the whole problem is
        parameterized by the random vector $\xi=(\xi_1,\dots, \xi_M)$
        and we can take as probability space the triplet:
        $\Gamma=[-1,1]^M$, $\mathcal{F}=\mathcal{B}(\Gamma)$ the Borel
        $\sigma$-algebra on $\Gamma$, and
        $\mathbb{P}[d\xi]=\otimes_{i=1}^M\frac{d\lambda(\xi_i)}{2}$
        the product uniform measure on $\Gamma$, $\lambda$ being the
        Lebesgue measure on $[-1,1]$.
	
	\subsection{Finite Element approximation}\label{FEsec_specific}
	To compute numerically an optimal control we consider a Finite Element
	(FE) approximation of the underlying PDE \eqref{eq:primal} and of the
	infinite dimensional OCP \eqref{eq:ocp}. We denote by $\{ \tau_h\}_{h>0}$ a family of regular triangulations of $D$ and choose $Y^h$ to
	be the space of continuous piece-wise polynomial functions of degree at most
	$r$ over $\tau_h$ that vanish on $\partial D$, i.e. $Y^h=\{ y \in
	C^0(\overline{D}): y\vert_K \in \mathbb{P}_r(K)\quad \forall K \in
	\tau_h, y \vert_{\partial D}=0 \} \subset Y$, and $U^h=Y^h$.
	For such spatial discretization of the OCP \eqref{eq:ocp}, the following error estimate has been obtained in \cite{MartinKrumscheidNobile}.
	\begin{theorem}
	\label{c1}
	Let $\uu$ be the optimal control, solution of problem \eqref{eq:ocp}, and denote by $\uuh$ the solution of the finite element approximate problem (see formulation \eqref{eq:FE_OCP}). Assume that: the domain $D \subset \R^d$ is polygonal convex; the random field $a$ satisfies \eqref{eq:uniform_a} and is such that $\nabla a \in L^{\infty}(D \times \Gamma)$; the primal and adjoint solutions in the optimal control $\uu$ satisfy $y(\uu), p(\uu) \in L_{\mathbb{P}}^2(\Gamma; H^{r+1}(D))$. Then
	\begin{multline}
	\label{febound}
	\| \uu-\uuh \|_{L^2(D)}^2+  \E[\| y(\uu)-y^h(\uuh) \|_{L^2(D)}^2]+h^2 \E[\| y(\uu)-y^h(\uuh) \|_{H^1_0}^2]\\ 
	\leq B_1 h^{2r+2} \{\E[|y(\uu)|^2_{H^{r+1}}]+\E[| p(\uu) |^2_{H^{r+1}}]\},
	\end{multline}
	with a constant $B_1$ independent of $h$.
	\end{theorem}
	Hence, such finite element approximation satisfies Assumption \ref{ass:FErate} with $p=r$.

	\subsection{Collocation on tensorized Gauss-Legendre points}
	\label{Sec2col}
	We approximate the expectation by a tensorized Gauss-Legendre
        quadrature formula. For $q\in\mathbb{N}_+$, let
        $\eta_1^{(q)},\ldots,\eta_{q}^{(q)}\in [-1,1]$ be the zeros of
        the Legendre polynomial of degree $q$ and
        $\zeta_1^{(q)},\ldots,\zeta_q^{(q)}\in\mathbb{R}_+$ the
        weights of the Gauss-Legendre quadrature formula (with respect
        to the Lebesgue measure). For a multi-degree
        $\mathbf{q}=(q_1,\ldots,q_M)\in\mathbb{N}^M_+$ {and multi-indices $\ii=(i_1,\ldots,i_M)\in \Xi := \prod_{k=1}^M \{1,\ldots,q_k\}$, we introduce
        the tensorized quadrature points
        $\{\eta_{\ii}=(\eta_{i_1}^{(q_1)},\ldots,\eta_{i_M}^{(q_M)})\}_{\ii\in\Xi}$ and weights
        $\{\zeta_{\ii}=2^{-M}\prod_{k=1}^M\zeta_{i_k}^{(q_k)}\}_{\ii\in\Xi}$.} Hence, for a continuous
        function $X: \Gamma \rightarrow \R$,
        $(\xi_1,\ldots,\xi_M)\mapsto X(\xi_1,\ldots,\xi_M)$, we
        approximate the expectation by the tensorized Gauss-Legendre
        quadrature formula
	\begin{equation} 
	\widehat{E}_\qq^{GL}[X]=\sum_{\ii\in\Xi} \zeta_\ii X(\eta_\ii),
	\end{equation}
	which can be written in the form \eqref{eq:quadrature} with $n=\prod_{k=1}^M q_k$ upon introducing a global numbering of the nodes.
	The following error estimate has been shown in \cite{MartinKrumscheidNobile}.
	\begin{lemma}\label{lem:quad}
	Let $\uu$ be the optimal control, solution of \eqref{eq:ocp}, and $\uhu$ the solution of the OCP problem approximated in probability by the quadrature formula $\widehat{E}_\qq^{GL}$ (see formulation \eqref{eq:semi_discret_gen}). Then there exists $B_2>0$ independent of $\qq$ s.t.
	\begin{equation} \label{eq:quad}
	\| \uu-\uhu \|^2+\widehat{E}[\| y(\uu)-y(\uhu)\|^2] \leq B_2 \|\E[p(\uhu)]-\widehat{E}[p(\uhu)] \|^2.
	\end{equation}
	\end{lemma}
	The convergence rate of the right hand side of \eqref{eq:quad}, as a function of $n$, depends on the smoothness of the function $\xi \mapsto p_{\xi}(\uhu)$.
	Following the arguments in \cite{NobileSto}, it can be shown that, under the assumption that the diffusion coefficient $\xi \mapsto a(\cdot, \xi) \in L^{\infty}(D)$ is analytic in each variable $(\xi_1, \cdots, \xi_M)$ in $\Gamma$, namely there exist $0< \gamma_1, \dots, \gamma_M \in \R$ and $B_3>0$ such that
	\begin{equation}\label{eq:analyticity_a}
	\left\Vert \frac{\partial^k a(\cdot, \xi)}{\partial \xi_j^k} \right\Vert_{L^{\infty}(D)} \leq B_3 k! \gamma_j^k,
	\end{equation}
	for any $u \in U$, the primal solution $\xi \mapsto y_{\xi}(u) \in H^1_0(D)$ as well as the adjoint solution $\xi \mapsto p_{\xi}(u) \in H^1_0(D)$ are both analytic in $\Gamma$ (see also \cite[Lemma 7]{MartinKrumscheidNobile}) and the following result holds:
	\begin{lemma} \label{th:coloc}
	For the OCP \eqref{eq:ocp} and its approximation by tensorized Gauss-Legendre quadrature formula $\widehat{E}_\qq^{GL}$, if \eqref{eq:analyticity_a} holds, then  there exist $0< s_1, \cdots, s_M \in \R$ independent of $\qq=(q_1,\ldots,q_M)$ s.t.
	\[
	\| \E[y(u)]-E^{GL}_\qq[y(u)] \|_{H^1_0(D)} + \| \E[p(u)]-E^{GL}_\qq[p(u)] \|_{H^1_0(D)} \leq B_4 \sum_{n=1}^{M} e^{-s_n q_n}, \quad \forall u\in U
	\]
	where the constant $B_4$ depends on $u$ but not on $\qq$.
	\end{lemma}
	The analyticity assumption \eqref{eq:analyticity_a} is satisfied for instance for a random field of the form $a(x,\xi) = \sum_{k=1}^M \xi_k \psi(x)$ or $a(x,\xi) = \exp\{\sum_{k=1}^M \xi_k \psi(x)\}$ (see also  \cite{ChkifaCohenSchwab2015} for further examples). 
	In particular, by taking $\qq=(q,\ldots,q)$ and combining the results of Lemmas \ref{th:coloc} and \ref{lem:quad} we have
	\[
	\| \uu-\uhu \|_{L^2(D)} \le {B_4} e^{-s_{min} \sqrt[M]{n}}, \quad s_{min}=\min_{i=1,\ldots,M} s_i,
	\]
	and Assumption \ref{ass:conv_prob} holds with $\theta(n) = {B_4} e^{-s_{min} \sqrt[M]{n}}$.

{
  	\section{Stochastic approximation methods for PDE-constrained OCPs} \label{sec4}
	
	We aim now at applying Stochastic Approximation (SA) techniques,
        and, in particular, the stochastic gradient (SG) and SAGA
        algorithms, to solve the discretized OCP. To keep the notation
        light, we present the different optimization algorithms and
        convergence estimates only for the semi-discrete problem
        \eqref{eq:semi_discret_gen}, although all results extend
        straightforwardly to the fully discrete case
        \eqref{eq:full_discret_gen}. This has also the advantage of
        ensuring that all constants and convergence rates in our
        estimates do not degenerate when the spatial discretization
        parameter $h$ goes to zero. Also, from now on, we will denote the norm $\|\cdot\|_U$ simply by $\|\cdot\|$ when no ambiguity arises.
	The objective function $\widehat{J}(u)$ in \eqref{eq:semi_discret_gen} reads
	\begin{align} \label{eqn:semiDOCP}
	\begin{split}
	\widehat{J}(u)&=\frac{1}{2} \widehat{E}[\| Q_\omega y_\omega (u)-z_d \|_H^2]+\frac{\beta}{2} \| u \|^2\\
	&=\sum_{i=1}^{n} \zeta_i f_i(u),
	\end{split}
	\end{align}
	with $f_i(u)=\frac{1}{2}\| Q_{\eta_i} y_{\eta_i}(u)-z_d
        \|_H^2+\frac{\beta}{2}\|u\|^2$ and $\nabla f_i(u)=\beta
        u+B_{\eta_i}^*p_{\eta_i}(u)$, where $\{ \eta_i \}_i$ are the nodes
        of the quadrature formula and $\{ \zeta_i \}_i$ the associated
        weights.

        The underlying idea in SA methods is that, at each iteration
        of the optimization loop, the full gradient
        $\widehat{J}(u)=\sum_{i=1}^{n} \zeta_i f_i(u)$ is never
        computed and only one or few terms in the sum, randomly drawn,
        are evaluated.  In equation \eqref{eqn:semiDOCP}, the
        functions $f_i(u)$ are naturally weighted by the possibly
        non-uniform weights $\{ \zeta_i \}_i$ and this raises the
        question whether the index $i_k$ of the term in the sum that
        is evaluated at iteration $k$ should be drawn from a
        \emph{uniform} or a \emph{non-uniform} distribution, possibly taking into account the quadrature weights $\{\zeta_i\}_i$. We
        take the second, more general, approach by introducing an auxiliary
        discrete probability measure $\widetilde{\zeta}$ on
        $\{1,\dots, n\}$, $\widetilde{\zeta}(j)=\widetilde{\zeta}_j>0,
        \mathrm{with} \sum_{j=1}^n \widetilde{\zeta}_j=1$ and using an
        \textit{importance sampling} strategy to evaluate the
        expectation. Namely, at iteration $k$ the gradient of
        $\widehat{J}$ at a given control $u$ is estimated by the
        single term $\frac{\zeta_{i_k}}{\widetilde\zeta_{i_k}}\nabla
        f_{i_k}(u)$, where $i_k$ is drawn from the distribution $\widetilde{\zeta}$, which is an unbiased estimator for
        $\nabla\widehat{J}(u)$, since
        \[
        \E[\frac{\zeta_{i_k}}{\widetilde\zeta_{i_k}}\nabla f_{i_k}(u)] = \sum_{p=1}^n \frac{\zeta_{p}}{\widetilde\zeta_{p}}\nabla f_{p}(u)\mathbb{P}[i_k=p] = \nabla\widehat{J}(u).
        \]
        With this in place, the \emph{stochastic gradient} method \emph{with importance sampling} (SG-IS), in its classical Robbins-Monro version \cite{robbins1951}, reads:

	

	\begin{equation} \label{modifiedSG}
	u_{k+1} = u_k - \tau_k \frac{\zeta_{i_k}}{\widetilde{\zeta}_{i_k}} \nabla f_{i_k}(u_k),
	\end{equation}
	where $i_k \sim \{\widetilde{\zeta}_{i} \}_i$ are i.i.d. discrete random variables on $\{ 1, \dots, n \}$, and the step size $\tau_k$ is chosen as $\tau_k=\frac{\tau_0}{k}$ with $\tau_0$ sufficiently large to guarantee converge of the iterations. Precise conditions on $\tau_0$  are given in Theorem \ref{thm:conv_SG} below.


        Similarly, the \emph{SAGA} method
        \cite{SchmidtLeRouxBach,DefazioBachLacosteJulien} \emph{with
        importance sampling} (SAGA-IS) reads:
	\begin{equation}
	\label{SAGAref}
	u_{k+1}=u_k-\tau \left(\left(\nabla f_{i_k}(u_k)-\nabla f_{i_k}(\phi_{i_k}^k)\right)\frac{\zeta_{i_k}}{\widetilde{\zeta}_{i_k}}+\sum_{j=1}^{n}\zeta_{j} \nabla f_j(\phi_j^k)\right)
	\end{equation}
	where again $i_k \stackrel{\text{\footnotesize iid}}{\sim} \{\widetilde{\zeta}_{i} \}_i$ and we set
	\begin{equation} \label{updatephi}
	\phi_{j}^{k+1}=
	\left\{
	\begin{array}{ll}
	u_k & \mathrm{if}\hspace{1mm} j = i_k, \\
	\phi_{j}^{k} & \mathrm{otherwise}. 
	\end{array}
	\right.
	\end{equation}
	The step size $\tau$ in \eqref{SAGAref} is typically kept
        fixed over the iterations and chosen sufficiently small to
        guarantee convergence. Precise conditions on $\tau$ will be
        given in Theorem \ref{generalRecu}.

	In practice, in the SAGA algorithm, we do not store the past controls $\phi_j^k$, rather the past gradients $grad^k_j=\nabla f_j(\phi_j^k)$. Similarly, we do not recompute at each iteration $k$ the whole sum $G_k=\sum_{j=1}^{n}\zeta_{j} \nabla f_j(\phi_j^k)$, rather update it using the formula 
	$$G_{k+1}=G_k-\zeta_{i_{k+1}} grad^k_{i_{k+1}}+\zeta_{i_{k+1}} \nabla f_{i_{k+1}}(u_k)$$ and then update the memory entries as:
	$$grad^{k+1}_{j}=
	\left\{
	\begin{array}{ll}
	\nabla f_j(u_k) & \mathrm{if}\hspace{1mm} j = i_{k+1}, \\
	grad^{k}_{j} & \mathrm{otherwise}. 
	\end{array}
	\right.$$

        We point out that both the SG and SAGA methods applied to the
        OCP \eqref{eq:full_discret_gen} require $2$ PDE solves per
        iteration. Moreover SAGA requires to store $1$ PDE solution
        and $n$ gradients at all iterations.

	Before analyzing the convergence and complexity of the SAGA-IS
        method, which is the main focus of this work, we briefly
        mention in the next section, the corresponding results for the
        SG-IS method.

	\subsection{Convergence and complexity analysis of the SG-IS algorithm}\label{sec:complexity_SG}
	
	Following the analysis in \cite{MartinKrumscheidNobile}, we
        can provide the following bound on the Mean Squared Error
        (MSE) of the SG-IS iterates \eqref{modifiedSG}, under the
        assumption that the importance sampling distribution
        $\{\widetilde\zeta_j\}_j$ is chosen so that the quantity
        \begin{equation}\label{defSn}
	\Stt=\sum_{j=1}^n\frac{\zeta^2_{j}}{\widetilde{\zeta}_{j}}.
	\end{equation}
        is uniformly bounded in $n$. If the weights are all positive,
        this can easily be achieved by taking for instance
        $\widetilde\zeta_j=\zeta_j$ which leads to $\Stt=1, \, \forall n$.
	\begin{theorem}\label{thm:conv_SG}
	  Let us denote by $\widehat{u}_{k}^h$ the $k$-th iterate of
          the SG-IS algorithm \eqref{modifiedSG}, applied to the fully
          discrete OCP \eqref{eq:full_discret_gen}, with
          $\tau_k=\frac{\tau_0}{k}$ and $\tau_0>\frac{1}{l}$, $l$
          being the strong convexity constant in \eqref{lemma:SC}, and
          by $\uu$ the solution of the original OCP
          \eqref{eq:weak_OCP}. Under Assumptions \ref{ass:general},
          \ref{ass:FErate}, \ref{ass:conv_prob}, and if there exist $
          \widetilde{S}>0$ s.t. for every $n \in \N$, $\Stt \leq
          \widetilde{S}$, the following bound on the MSE $\E[\|
            \widehat{u}_{k}^h-\uu \|^2]$ holds:
	\begin{equation} \label{3tSG}
	 \E[\| \widehat{u}_{k}^h-\uu \|^2] \leq C_1 k^{-1} + C_2\theta(n)^2 +C_3 h^{2p+2}\;,
	\end{equation}
	with constants $C_1$, $C_2$ and $C_3$ independent of $h, n$ and $k$, and $p$, $\theta(\cdot)$ from Assumptions \ref{ass:FErate} and \ref{ass:conv_prob}.
	\end{theorem}
	
	We omit the proof as it follows very similar steps as in
        \cite{MartinKrumscheidNobile}. We now analyze the complexity
        of the SG-IS algorithm \eqref{modifiedSG} in terms of
        computational work $W$ versus accuracy $tol$ in case of a
        sub-exponential decay of the quadrature error as for the
        problem is Section \ref{sec2}.
        To define the computational work $W$ we consider a reasonable
        \emph{computational work model}: we assume that each primal
        and adjoint problem, discretized using a triangulation with
        mesh size $h$, can be solved in computational time
        $C_h=O(h^{-d \gamma})$, where $\gamma \in [1,3]$ is a
        parameter representing the efficiency of the linear solver
        used (e.g. $\gamma=3$ for a direct solver and $\gamma=1$ up to
        a logarithm factor for an optimal multigrid solver), while $d$
        is the dimension of the physical space.  Therefore, the
        computational work of $k$ iterations of the SG-IS algorithm
        \eqref{modifiedSG} will be $O(kh^{-d \gamma})$.  This work
        model does not consider possible gains due to parallelization
        and high performance computing. Also, we neglect the cost of
        sampling the index $i_k$ from the distribution $\{\tilde
        \zeta_j\}_j$ as it is, in general, marginal with respect to
        the cost of solving the primal/adjoint PDE, even for large
        $n$.

	\begin{corollary}\label{corr:complexity1}
	Let the same hypotheses as in Theorem \ref{thm:conv_SG} hold,
        and assume further a quadrature error of the form
        $\theta(n)=C_4 e^{-s\sqrt[M]{n}}$ for some $s>0$ and
        $M\in\mathbb{N}_+$. In order to guarantee a MSE $\E[\|
          \widehat{u}_{k}^h-\uu \|^2] \lesssim tol^2$, the total
        required computational work using the SG-IS algorithm
        \eqref{modifiedSG} is bounded by
	\begin{equation}W  \lesssim tol^{-2-\frac{d \gamma}{r+1}}.
	\end{equation}
        Moreover the memory space required to store the gradient and
        the optimal control at each iteration scales as
	\begin{equation}
	\text{storage } \lesssim tol^{-\frac{d}{r+1}}
	\end{equation}
	\end{corollary}
	\begin{proof}
	If we want to guarantee a MSE error of size $O(tol^2)$, we can equalize the three terms on the right hand side of \eqref{3tSG} to $tol^2$ thus obtaining:
	\[ h\sim tol^{\frac{1}{r+1}}, \quad n\sim \left( \log(tol^{-1}) \right)^M, \quad k\sim tol^{-2}. \]
        On the other hand, the total computational work is
	\begin{equation}
	W=2C_h k \lesssim tol^{-2} tol^{-\frac{d \gamma}{r+1}}.
	\end{equation}
	The memory space required at each iteration corresponds to storing  one gradient and one control and is proportional to $h^{-d}$ thus leading to:
	\begin{equation}
	storage\lesssim tol^{-\frac{d}{r+1}}.
	\end{equation}
	\end{proof}
	Notice that the above (asymptotic) complexity result is
        independent of $n$ and $M$, having neglected the cost of sampling from the discrete distribution $\{\tilde\zeta_j\}_j$.

	\subsection{Convergence analysis of the SAGA-IS algorithm} \label{sec:SAGA_conv}
	
	The mean squared error of the SG-IS algorithm analyzed in the
        previous section decays at an algebraic rate $k^{-1}$ in the
        number of iterations with constant independent of $h$ and $n$
        under the assumptions of Theorem \ref{thm:conv_SG}. We show in
        this section that, under similar assumptions, the mean squared
        error of the SAGA-IS algorithm \ref{SAGAref} decays at
        exponential rate $(1-\epsilon)^k$ in the number of iterations,
        with $\epsilon \sim n^{-1}$. The result is given in Theorem
        \ref{generalRecu} and Corollary \ref{co:SAGA}. The outcome of
        our analysis in that uniform sampling of the index $i_k$
        (i.e. $\widetilde{\zeta}_j=\frac{1}{n}$, $\forall j=1, \dots,
        n$) is indeed asymptotically optimal in the sense that it
        provides the best convergence rate for large $n$ (see Remark
        \ref{rem:optimality}).
	The proof is inspired by \cite{DefazioBachLacosteJulien} and
        is valid under the general assumptions that: i) each $f_i$
        satisfies a Lipschitz property \eqref{lemma:Lip} with the same
        Lipschitz constant $L$ independent of $n$ and $h$, which is
        guaranteed for both the semi-discrete OCP
        \eqref{eq:semi_discret_gen} and the fully discrete OCP
        \eqref{eq:full_discret_gen}; ii) $\widehat{J}(u)=\sum_{i=1}^n
        \zeta_i f_i(u)$ is strongly convex with constant $l$
        independent of $n$ and $h$, which is guaranteed, for instance,
        if the weights $\{\zeta_j\}_j$ are all positive, or if $n$ is
        large enough (since $\widehat{J}(u)\to J(u)$ as $n\to\infty$
        and $J(u)$ is strongly convex); iii) the quantity
        $\Stt=\sum_{j=1}^n n\zeta_j^2$ is uniformly bounded in $n$.

        In what follows, we denote by $F_k$ the $\sigma$-algebra
        generated by the random variables $i_0, i_1, \dots, i_{k-1}$
        and denote by $\E[\cdot|F_k]$ the conditional expectation to
        such $\sigma$-algebra. Observe, in particular, that $u_k$ and
        $\{\phi^k_j\}_j$ are all $F_k$-measurable, so that
        $\E[\psi(u_k,\phi^k_1,\ldots,\phi^k_n)|F_k]=
        \psi(u_k,\phi^k_1,\ldots,\phi^k_n)$ for any measurable
        function $\psi:U^{n+1}\to \mathbb{R}$. Moreover, in the
        remaining of this Section, we use the shorthand notation
        $f'_j(u)$ to denote $\nabla f_j(u)$. For the convergence proof of SAGA-IS, we
        also need to introduce the quantity
	\begin{equation*}
	Q_k=\sum_{j=1}^{n}\frac{\zeta_j^2}{\widetilde{\zeta}_{j}} \|f'_j(\phi_j^k)-f'_j(\uhu)\|^2
	\end{equation*}
	where $\uhu$ denotes, as usual, the optimal control, solution of the semi-discrete OCP \eqref{eq:semi_discret_gen}. We start our convergence analysis by one technical Lemma.
	\begin{lemma} 
	\label{trick_split}
	Let $\{\g_j:U\to\mathbb{R}\}_{j=1}^n$ be a given collection of continuous functions. Then 
	\begin{equation}
	\E \left[\sum_{j=1}^{n}\frac{\zeta^2_{j}}{\widetilde{\zeta}_{j}}\g_j(\phi_{j}^{k+1}) | F_k \right]=\sum_{j=1}^{n}\frac{\zeta^2_{j}}{\widetilde{\zeta}_{j}} (1-\widetilde{\zeta}_{j})\g_{j}(\phi_{j}^{k})+\sum_{p=1}^n{\zeta^2_{p}}\g_{p}(u_k)
	\end{equation}
	\end{lemma}
	\begin{proof}
	Using the law of total probability we write the conditional expectation as a sum over the possible values $i_k=p$, $p \in \{ 1, \dots, n \}$; 
	\begin{align*}
	\E \left[\sum_{j=1}^{n}\frac{\zeta^2_{j}}{\widetilde{\zeta}_{j}}\g_j(\phi_{j}^{k+1}) | F_k \right]&=\sum_{p=1}^n \E \left[ 1_{[ i_k=p]}\sum_{j=1}^{n}\frac{\zeta^2_{j}}{\widetilde{\zeta}_{j}}\g_j(\phi_{j}^{k+1}) | F_k \right] \\
	&=\sum_{p=1}^n \E \left[\sum_{j=1}^{n}\frac{\zeta^2_{j}}{\widetilde{\zeta}_{j}}\g_j(\phi_{j}^{k+1}) | F_k, i_k=p \right]\mathbb{P}[i_k=p] \\
	&=\sum_{p=1}^n \widetilde{\zeta}_{p} \left \{ \sum_{j=1, j \neq p}^{n}\frac{\zeta^2_{j}}{\widetilde{\zeta}_{j}}\g_{j}(\phi_{j}^{k})+\frac{\zeta^2_{p}}{\widetilde{\zeta}_{p}}\g_p(u_k)\right \}\\
	&=\sum_{p=1}^n \widetilde{\zeta}_{p} \left \{ \sum_{j=1}^{n}\frac{\zeta^2_{j}}{\widetilde{\zeta}_{j}}\g_{j}(\phi_{j}^{k})\right \}-\sum_{p=1}^n  {\zeta^2_{p}}\g_{p}(\phi_{p}^{k})+\sum_{p=1}^n{\zeta^2_{p}}\g_{p}(u_k) \\
	&=\sum_{j=1}^{n}\frac{\zeta^2_{j}}{\widetilde{\zeta}_{j}} (1-\widetilde{\zeta}_{j})\g_{j}(\phi_{j}^{k})+\sum_{p=1}^n{\zeta^2_{p}}\g_{p}(u_k)\\
	\end{align*}
	\end{proof}
	Now we use the latter technical Lemma to prove a bound on $\E[Q_{k+1}|F_k]$. This is the purpose of the next Lemma.
	
	\begin{lemma} \label{qk}
	We have the following bound on the conditional expectation $\E[Q_{k+1}|F_k]$:
	\begin{equation}
	\E[Q_{k+1}|F_k] \leq \max_j (1-\widetilde{\zeta}_{j}) Q_k+ \St L^2 \| u_k-\uhu \|^2 \hspace{2mm} \mathrm{with} \hspace{1mm} \St=\sum_{p=1}^{n} \zeta_p^2
	\end{equation}
	\end{lemma}
	\begin{proof}
	Using Lemme \ref{trick_split} for $\g_j(\cdot)=\|f'_{j}(\cdot)-f'_{j}(\uhu)\|^2$, we have
	\begin{align*}
	\E[Q_{k+1}|F_k]&=\sum_{j=1}^{n}\frac{\zeta^2_{j}}{\widetilde{\zeta}_{j}} (1-\widetilde{\zeta}_{j})\|f'_{j}(\phi_{j}^{k})-f'_{j}(\uhu)\|^2+\sum_{p=1}^n{\zeta^2_{p}}\|f'_{p}(u_k)-f'_{p}(\uhu)\|^2\\
	&\leq \max_j \left( 1-\widetilde{\zeta}_{j} \right) \sum_{j=1}^{n}\frac{\zeta^2_{j}}{\widetilde{\zeta}_{j}} \|f'_{j}(\phi_{j}^{k})-f'_{j}(\uhu)\|^2+\sum_{p=1}^n{\zeta^2_{p}}L^2 \| u_k-\uhu \|^2\\
	&\leq \max_j \left( 1-\widetilde{\zeta}_{j} \right) Q_k+\St L^2 \| u_k-\uhu \|^2
	\end{align*}
	\end{proof}
	\begin{lemma} \label{props}
	Let $P_k=\left(f'_{i_k}(u_k)-f'_{i_k}(\phi_{i_k}^k)\right)\frac{\zeta_{i_k}}{\widetilde{\zeta}_{i_k}}+\sum_{j=1}^{n}f'_j(\phi_j^k)\zeta_j$ and $T_k=P_k-\nabla  \widehat{J}(\uhu)$, then we have the following properties:
	\begin{equation}\label{P1}
	\E[P_k|F_k]=\nabla \widehat{J}(u_k)
	\end{equation}
	\begin{equation}\label{P2}
	\E[T_k|F_k]=\nabla \widehat{J}(u_k)-\nabla \widehat{J}(\uhu)
	\end{equation}
	\begin{equation}\label{P3}
	\E[\| T_k\|^2|F_k] \leq 9 \Stt L^2\| u_k-\uhu \|^2 +8Q_k 
	\end{equation}
	where $\Stt$ is defined in \eqref{defSn}.
	\end{lemma}
	\begin{proof}
	Again, we further condition on the possible values taken by the random variable $i_k$, thus obtaining:
	\begin{align*}
	\E[P_k|F_k]&=\E\left[\left(\left(f'_{i_k}(u_k)-f'_{i_k}(\phi_{i_k}^k)\right)\frac{\zeta_{i_k}}{\widetilde{\zeta}_{i_k}}+\sum_{j=1}^{n}f'_j(\phi_j^k)\zeta_j\right)|F_k\right]\\
	&=\sum_{p=1}^n \E\left[\left(\left(f'_{i_k}(u_k)-f'_{i_k}(\phi_{i_k}^k)\right)\frac{\zeta_{i_k}}{\widetilde{\zeta}_{i_k}}+\sum_{j=1}^{n}f'_j(\phi_j^k)\zeta_j\right)|F_k, i_k=p\right]{\widetilde{\zeta}_{p}}\\
	&= \sum_{j=1}^{n} f'_{j}(u_k)\frac{\zeta_{j}}{\widetilde{\zeta}_{j}}  {\widetilde{\zeta}_{j}}-\sum_{j=1}^{n} f'_{j}(\phi_j^k)\frac{\zeta_{j}}{\widetilde{\zeta}_{j}}  {\widetilde{\zeta}_{j}}+\sum_{j=1}^{n}f'_j(\phi_j^k)\zeta_j\\
	&= \sum_{j=1}^{n} f'_{j}(u_k) {\zeta_{j}} = \nabla \widehat{J}(u_k)
	\end{align*}
	which proves \eqref{P1}. We see from this that $P_k$ is an unbiased estimator of $\nabla \widehat{J}(u_k)$, when conditioned to $F_k$.
	Equation \eqref{P2} follows straightforwardly:
	\[\E[T_k|F_k]=\nabla \widehat{J}(u_k)-\nabla \widehat{J}(\uhu).\]
	We prove now \eqref{P3}.
	\begin{align*}
	\E[\| T_k\|^2|F_k]=& \E[\| T_k-\E[T_k|F_k]\|^2|F_k]+\|\E[ T_k|F_k]\|^2\\
	=&\E[\|P_k-\nabla \widehat{J}(u_k)\|^2|F_k]+\| \nabla \widehat{J}(u_k)-\nabla \widehat{J}(\uhu) \|^2\\
	=&\E[\|\left(f'_{i_k}(u_k)-f'_{i_k}(\phi_{i_k}^k)\right)\frac{\zeta_{i_k}}{\widetilde{\zeta}_{i_k}}+\sum_{j=1}^{n}f'_j(\phi_j^k)\zeta_j-\sum_{j=1}^{n}f'_j(u_k)\zeta_j\|^2|F_k]\\&+\| \nabla \widehat{J}(u_k)-\nabla \widehat{J}(\uhu) \|^2\\
	 \leq &2\underbrace{\E[\|\left(f'_{i_k}(u_k)-f'_{i_k}(\phi_{i_k}^k)\right)\frac{\zeta_{i_k}}{\widetilde{\zeta}_{i_k}}\|^2|F_k]}_{=:A}\\&+2\underbrace{\E[\|\sum_{j=1}^{n}\zeta_j \left(f'_j(\phi_j^k)-f'_j(u_k)\right)\|^2|F_k]}_{=:B}+\underbrace{\| \nabla \widehat{J}(u_k)-\nabla \widehat{J}(\uhu) \|^2}_{=:C}
	\end{align*}
	The first part $A$ can be split as
	\begin{align*}
	A &=\E[\|\underbrace{\left(f'_{i_k}(u_k)-f'_{i_k}(\phi_{i_k}^k)\right)}_{\pm f'_{i_k}(\uhu)}\frac{\zeta_{i_k}}{\widetilde{\zeta}_{i_k}}\|^2|F_k]\\
	& \leq 2\underbrace{\E[\|\left(f'_{i_k}(u_k)-f'_{i_k}(\uhu)\right)\frac{\zeta_{i_k}}{\widetilde{\zeta}_{i_k}}\|^2|F_k]}_{=:T_1}+2\underbrace{\E[\|\left(f'_{i_k}(\uhu)-f'_{i_k}(\phi_{i_k}^k)\right)\frac{\zeta_{i_k}}{\widetilde{\zeta}_{i_k}}\|^2|F_k]}_{=:T_2}
	\end{align*}
	with
	\[T_1 \leq  L^2\| u_k-\uhu \|^2 \E[\frac{\zeta^2_{i_k}}{\widetilde{\zeta}^2_{i_k}}]= L^2\| u_k-\uhu \|^2 \Stt\]
	The term $T_2$ can be developed as a sum over the possible values of $i_k$:
	\[T_2=\E[\frac{\zeta^2_{i_k}}{\widetilde{\zeta}^2_{i_k}}\|f'_{i_k}(\phi_{i_k}^k)-f'_{i_k}(\uhu)\|^2|F_k]=\underbrace{\sum_{j=1}^{n}\frac{\zeta^2_{j}}{\widetilde{\zeta}_{j}}\|f'_{j}(\phi_{j}^k)-f'_{j}(\uhu)\|^2}_{=:Q_k}\]
	Moreover 
	\begin{align*}
	B&=\E[\|\sum_{j=1}^{n}\zeta_j \left(f'_j(\phi_j^k)-f'_j(u_k)\right)\|^2 |F_k] \\
	& \le \left(\sum_{j=1}^{n}|\zeta_j| \, \|f'_j(\phi_j^k)-f'_j(u_k)\|\right)^2\\
	& \leq \left(\sum_{j=1}^{n}\frac{\zeta_j^2}{\widetilde{\zeta}_{j}} \|\underbrace{f'_j(\phi_j^k)-f'_j(u_k)}_{\pm f'_j(\uhu)}\|^2\right)\underbrace{\left(\sum_{j=1}^{n}\widetilde{\zeta}_{j}\right)}_{=1}\\
	& \leq 2\sum_{j=1}^{n}\frac{\zeta_j^2}{\widetilde{\zeta}_{j}} \|f'_j(\phi_j^k)-f'_j(\uhu)\|^2+2\sum_{j=1}^{n}\frac{\zeta_j^2}{\widetilde{\zeta}_{j}} \|f'_j(u_k)-f'_j(\uhu)\|^2\\
	& \leq 2 Q_k+2L^2 \Stt \| u_k-\uhu \|^2\\ 
	\end{align*}
	Finally
	\begin{align*}
	C&=\| \sum_{p=1}^{n} \zeta_p \left(  f'_p(u_k)-f'_p(\uhu)  \right) \|^2\\
	&\leq \left( \sum_{p=1}^{n} |\zeta_p| \|  f'_p(u_k)-f'_p(\uhu)   \|\right)^2\\
	&\leq L^2 \|  u_k-\uhu  \|^2 \left( \sum_{p=1}^{n} |\zeta_p| \right)^2\\
	&\leq \Stt L^2 \|  u_k-\uhu  \|^2
	\end{align*}
	which completes the proof.
	\end{proof}
	
	\begin{lemma} \label{SAGArecu}
          Suppose that the quadrature weights in \eqref{eq:quadrature} are all positive, $\zeta_j >0$. For any $\alpha>0$ and any step size $\tau>0$, if $u_k$ denotes the $k$-th iterate of SAGA-IS algorithm \ref{SAGAref} and $\uhu$ is the solution of the semi-discrete OCP \eqref{eq:semi_discret_gen}, then
	\[
	\E[\|u_{k+1}-\uhu  \|^2+\alpha Q_{k+1}|F_k] \leq D_1 \|u_{k}-\uhu\|^2+ D_2 \alpha Q_k
	\]
	with $D_1=1-l \tau +(\alpha \St+9 \tau^2 \Stt) L^2$, $D_2=1-\widetilde{\zeta}_{min}+\frac{8\tau^2}{\alpha}$, $\widetilde{\zeta}_{min}=\min_j \widetilde{\zeta}_{j}$, $\St$ as in Lemma \ref{qk} and $\Stt$ as in equation \eqref{defSn}.
	\end{lemma}
	\begin{proof}
	By optimality, we have that $\nabla  \widehat{J}(\uhu)=0$,which implies
	\[\| u_{k+1}-\uhu \|^2=\| u_{k}-\uhu- \tau\underbrace{(P_k-\nabla  \widehat{J}(\uhu))}_{=T_k} \|^2=\| u_{k}-\uhu \| ^2 -2\tau \langle u_{k}-\uhu ,T_k \rangle+ \tau^2\|T_k \|^2.\]  
	Let us develop now $\|u_{k+1}-\uhu  \|^2+ \alpha Q_{k+1}$, 
	using Lemmas \ref{qk} and \ref{props},
	\begin{align*}
	\E[\|&u_{k+1}-\uhu  \|^2+ \alpha Q_{k+1}|F_k] \\&= \E[\|u_{k}-\uhu\|^2|F_k]-2\tau \E[\langle u_{k}-\uhu ,T_k \rangle|F_k]+ \tau^2\E[\|T_k \|^2|F_k]+\alpha \E[Q_{k+1}|F_k]\\
	&= \|u_{k}-\uhu\|^2-2\tau \langle u_{k}-\uhu ,\nabla \widehat{J}(u_k)-\nabla \widehat{J}(\uhu) \rangle+ \tau^2\E[\|T_k \|^2|F_k]+\alpha \E[Q_{k+1}|F_k]\\
	&\leq  \|u_{k}-\uhu\|^2-l\tau \|u_{k}-\uhu\|^2+\tau^2\left( 
	9 \Stt L^2\| u_k-\uhu \|^2 +8Q_k 
	\right) \\
	&+\alpha \left(\max_j (1-\widetilde{\zeta}_{j}) Q_k+\St L^2 \| u_k-\uhu \|^2 \right)\\
	& \leq \left( 1-l \tau +(\alpha \St +9 \tau^2 \Stt) L^2 \right) \|u_{k}-\uhu\|^2 +  \left(\max_j (1-\widetilde{\zeta}_{j})+8\frac{\tau^2}{\alpha}\right)\alpha Q_k 
	\end{align*}
	\end{proof}
	We are now ready to state the final convergence result. For this, we need to find the right choice of $\alpha>0$ and $\tau$ s.t.
	$D_1<1$ and $D_2<1$ with $D_1$ and $D_2$ defined in Lemma \ref{SAGArecu}.
	One particular condition that guarantees an exponential in $k$ convergence rate is shown in the following Theorem.
	
	\begin{theorem}\label{generalRecu} Let $u_k$ denote the $k$-th iterate of SAGA-IS algorithm \ref{SAGAref} and $\uhu$ denote the solution of the semi-discrete OCP \eqref{eq:semi_discret_gen}. Assuming positive quadrature weights $\zeta_j >0$, for every $0<\tau < \tau_1:=\frac{l}{25   \Stt L^2}$ and setting
	\[ \widetilde{\zeta}_{j}=\frac{1}{n}, \quad \alpha=16n\tau^2\]
	we have
	\[
	\E[\|u_{k+1}-\uhu  \|^2+\alpha Q_{k+1}] \leq \rho(\tau) \E[ \|u_{k}-\uhu  \|^2+\alpha Q_{k}, ]
	\]
	with $\rho(\tau)=\min\{D_1(\tau),1-\frac{1}{2n}\} <1$ and $D_1(\tau):= 1-l \tau +(\alpha \St +9 \tau^2 \Stt) L^2 $.
	\end{theorem}

	\begin{proof}
	Using Lemma \ref{SAGArecu}, we have:
	\begin{align*}
	\E[\|u_{k+1}-\uhu  \|^2+\alpha Q_{k+1}|F_k] &\leq D_1 \|u_{k}-\uhu\|^2+ D_2 \alpha Q_k \\&\leq \underbrace{\max(D_1(\tau), D_2)}_{= \rho} \left( \|u_{k}-\uhu  \|^2+\alpha Q_{k} \right).
	\end{align*}
	with $D_2=1-\frac{1}{2n}$. The condition $0<\tau<\tau_1$ guarantees that $0<D_1(\tau)<1$. The final result is obtained by taking a further expectation over the r.v. $(i_0, \dots, i_{k-1})$ and using the law of total expectation.
	\end{proof}
	
	\begin{corollary} \label{co:SAGA}
        Under the same hypotheses of Theorem \ref{generalRecu}, if
        there exists $ \widetilde{S}>0$ s.t. for every $n \in \N$,
        $\Stt \leq \widetilde{S}$ and choosing $\tau=\tau_1/2$, then
	\begin{equation}\label{SAGArate}
	 \E\left[\| u_k-\uhu \|^2\right] \leq E_1 (1-\epsilon)^k  \hspace{2mm} \mathrm{with} \hspace{1mm} \epsilon=\min\left\{\frac{l^2}{100 \Sti L^2},\frac{1}{2n}\right\},
	 \end{equation}
        with  $E_1>0$ independent of $n$ and $k$.
	\end{corollary}

	\begin{proof}
	When $\tau=\tau_1/2$, the term $D_1(\tau)$ becomes minimal and the expressions of $\{ \widetilde{\zeta}_j \}_j$, $\alpha$ and $\tau$ imply
	$$D_1 = 1-\frac{l^2}{100 \Stt L^2}\leq 1-\frac{l^2}{100 \widetilde{S} L^2} \;\; \text{ and } \;\;D_2=1-\frac{1}{2n},$$ 
	 where we have exploited the fact that $n \St= \Stt$ for $\widetilde{\zeta}_{j}=1/n$.
	The final result is a direct application of Theorem \ref{generalRecu}.
	\end{proof}
	  
        The convergence results stated in Theorem \ref{generalRecu}
        and Corollary \ref{co:SAGA} for the semi-discrete OCP
        \eqref{eq:semi_discret_gen} apply equally well to the fully
        discrete OCP \eqref{eq:full_discret_gen} with the same
        constants, thanks to the fact that the spatially approximated
        functions $f^h_i(\cdot)$ satisfy the strong convexity and
        Lipschitz properties \eqref{lemma:Lip}, \eqref{lemma:SC} with
        the same constants as for $f_i(\cdot)$.
	
	The condition $\Stt\le\widetilde{S}, \, \forall
        n\in\mathbb{N}_+$ in Corollary \ref{co:SAGA} might seem very
        restrictive. However, it is trivially satisfied for Monte
        Carlo and Quasi Monte Carlo quadrature formulas, for which
        $\zeta_j=\frac{1}{n}$ and we show in the next Lemma that such
        condition holds also for tensorized Gauss-Legendre quadrature
        formulas.
	\begin{lemma} \label{Szego}
	In the setting of a probability space parametrized by $M$ iid
        uniform random variables $(\xi_1,\ldots,\xi_M)$, and
        $\widehat{E}$ being a tensorized Gauss-Legendre quadrature formula,
        when choosing $\widetilde{\zeta}_{j}=\frac{1}{n}, \;
        j=1,\ldots,n$, there exists $ \widetilde{S}>0$ s.t. for every
        $n \in \N$, $\Stt \leq \widetilde{S}$. More generally, the
        result holds true for $\xi_k$ having beta distribution and
        using corresponding tensorized Gauss-Jacobi quadrature
        formulas.
	\end{lemma}
	\begin{proof}
	As shown in \cite[page 353, (15.3.10)]{szego1939orthogonal}, the weights of the uni-variate Gauss-Legendre (resp. Gauss-Jacobi) quadrature formula with $q$ points satisfy
        \[
        \zeta_{i}^{(q)}\lesssim \frac{1}{q}, \quad \forall i=1,\ldots,q.
        \]
        Hence, for a tensor quadrature formula with $(q_1, \dots, q_M)$
        points in each variable and $n=\prod_k q_k$ points in
        total, and a multi-index $\ii=(i_1, \dots, i_M)$, with $1
        \leq i_k \leq q_k$, we have
	\[
        \zeta_\ii=\prod_{k=1}^{M} \zeta_{i_k}\lesssim \prod_{k=1}^M \frac{1}{q_k}
        \]
	and
	\begin{align*}
	\Stt = \sum_\ii n \zeta_\ii^2 &\lesssim \sum_{i_1=1}^{q_1}\dots \sum_{i_M=1}^{q_M} \left( \prod_{k=1}^{M}\frac{1}{q_k} \right)^2 \prod_{q=1}^{M} q_k\\
	&=\sum_{i_1=1}^{q_1}\dots \sum_{i_M=1}^{q_M}  \prod_{k=1}^{M}\frac{1}{q_k} = 1
	\end{align*}
	with hidden constant independent of $(q_1, \dots, q_M)$, but depending exponentially on~$M$.
	\end{proof}

        \begin{remark}[On negative quadrature weights]\label{rem:nonpositive_weights}
          The results of Lemma \ref{SAGArecu} and Theorem
          \ref{generalRecu} still hold also for quadrature weights
          that might have alternating sign, as long as the functional
          $\widehat{J}$ remains strongly convex, i.e. there exists
          $\tilde{l}>0$ such that $\frac{\tilde l}{2} \|u_1-u_2\|^2
          \le \langle \nabla \widehat{J}(u_1)- \nabla
          \widehat{J}(u_2), u_1-u_2\rangle$, for any $u_1,u_2\in
          U$. This will be guaranteed, in general, for $n$ large
          enough and with constant $\tilde l$ close to $l$ since
          $\widehat{J}(u) \to J(u)$ as $n\to\infty$ and $J$ is
          strongly convex with constant $l$.  On the other hand, the
          condition $\Stt\le \widetilde{S}, \forall n\in\mathbb{N}_+$
          will be more difficult to satisfy or might not hold at
          all. If for instance one has $\Stt=O(n^\alpha)$,
          for some $\alpha>1$, Corollary \ref{co:SAGA} will predict
          exponential convergence with a worse asymptotic rate
          $\epsilon = O(n^{-\alpha})$.
        \end{remark}

	\begin{remark}
        In the original paper \cite{SchmidtLeRouxBach}, in which the
        SAGA method was proposed, the authors have considered the
        minimization problem
        \[
        \uhu\in\argmin_u J(u) = \sum_{i=1}^n \frac{1}{n}g_i(u)
        \]
        where all the functions $g_i, \; i=1,\ldots,n$ satisfy strong
        convexity and Lipschitz properties \eqref{lemma:Lip},
        \eqref{lemma:SC} with the same constants $\hat{l}$ and
        $\hat{L}$, independent of $n$.
        Under these assumptions, they have derived the convergence estimate
        \[
          \E[\|u_k-\uhu\|^2] \lesssim (1-\hat{\epsilon})^k, \qquad \text{with } \hat{\epsilon} = \min\left(\frac{\hat{l}}{16 \hat{L}},\frac{1}{8n}\right)
        \] 
        when a fixed step size $\tau=\frac{1}{16\hat{L}}$ is used. 

        In our setting, $g_i(u) = \frac{1}{n}\zeta_i f_i(u)$ where all
        the functions $f_i$ satisfy strong convexity and Lipschitz
        properties with constants $l$ and $L$. Because of the non
        uniform weights $\zeta_i$, however, the convexity and
        Lipschitz constants for the corresponding functions $g_i$ are
        now $n$-dependent and read
        \[
        \hat{l} = \frac{\zeta_{min} l}{n}, \qquad \hat{L}=\frac{\zeta_{max} L}{n}
        \]
        where we have assumed positive quadrature weights. This implies
        \[
        \hat{\epsilon}=\min\left(\frac{\zeta_{min}}{\zeta_{max}}\frac{l}{16 L},\frac{1}{8n}\right),
        \]
        which has to be compared with the result in Corollary \ref{co:SAGA}.
        Depending on the quadrature formula used, the ratio
        $\frac{\zeta_{min}}{\zeta_{max}}$ might behave much worse than
        $n^{-1}$, whereas, with our SAGA-IS algorithm we can guarantee a rate $\epsilon\sim n^{-1}$ provided the condition $\Stt\le\widetilde{S}, \, \forall n$ holds.
	\end{remark}

        \begin{remark}[optimality of uniform IS measure]\label{rem:optimality}
          From Lemma \ref{SAGArecu}, we infer that 
          \[
	    \E[\|u_{k+1}-\uhu  \|^2+\alpha Q_{k+1}] \leq \max\{D_1,D_2\} \E[\|u_{k}-\uhu\|^2 + \alpha Q_k].
	  \]
          To achieve $\max\{D_1,D_2\} = (1-\epsilon)$ with $\epsilon\sim n^{-1}$, a necessary condition is that $1-D_2 = \widetilde{\zeta}_{min} - \frac{8\tau^2}{\alpha} \sim n^{-1}$, which implies $\widetilde{\zeta}_{min} \sim n^{-1}$. In particular, $D_2$ is minimized when $\widetilde{\zeta}_{min}=\frac{1}{n}$, i.e. when $\widetilde{\zeta}$ is the uniform distribution on $\{1,\ldots,n\}$. We already know from Theorem \ref{generalRecu} that, for such choice of importance sampling measure, the term $D_1$ can be made independent of $n$, so the uniform IS measure is asymptotically optimal as $n\to \infty$. 
        \end{remark}

	\subsection{Complexity analysis of the SAGA-IS algorithm}\label{sec:complexity}

        In the previous section we have focused on the convergence of
        the SAGA-IS method to the optimal control, solution of the
        semi-discrete OCP \eqref{eq:semi_discret_gen} or the fully
        discrete OCP \eqref{eq:full_discret_gen}.  We consider now
        the SAGA-IS method applied to the fully discrete OCP
        \eqref{eq:full_discret_gen} and analyze its convergence to the
        ``true'' optimal control, solution of the original OCP
        \eqref{eq:weak_OCP}, investigating all three sources of error,
        namely: spatial discretization, quadrature error and finite
        number of SAGA iterations.
        The following result holds.
	
	\begin{theorem} \label{3T}
        Let $\uS$ denote the $k$-th iterate of the
        SAGA-IS algorithm \eqref{SAGAref}, applied to the fully
        discrete OCP \eqref{eq:full_discret_gen}, with $\tau$ and
        $\{\tilde\zeta_j\}_j$ chosen as in Corollary \ref{co:SAGA} and
        let $\uu$ denote the solution of the original OCP
        \eqref{eq:weak_OCP}. Under Assumptions \ref{ass:general},
        \ref{ass:FErate}, \ref{ass:conv_prob}, and if there exists $
        \widetilde{S}>0$ s.t. $\Stt \leq \widetilde{S}, \, \forall
        n\in\mathbb{N}_+$, the following bound on the MSE $\E[\|
          \uS-\uu \|^2]$ holds:
	\begin{equation} \label{3t}
	 \E[\| \uS -\uu \|^2] \leq E_1(1-\epsilon)^k + E_2\theta(n)^2 +E_3 h^{2p+2}\;,
	\end{equation}
	with constant $E_1$, $E_2$ and $E_3$ independent of $h, n$ and $k$, with $\epsilon$
        as in Corollary \ref{co:SAGA} and with $p$,
        $\theta(\cdot)$ from Assumptions \ref{ass:FErate} and
        \ref{ass:conv_prob}.
	\end{theorem}
	
	\begin{proof}
	We can decompose the total error in the three error contributions corresponding to the spatial discretization, the quadrature and the SAGA optimization procedure:
	\begin{equation}
	\E[\| \uS-\uu \|^2]\leq 3\underbrace{\E[\| \uS-\uhhu \|^2]}_{SAGA}+3\underbrace{\| \uhhu-\uuh \|^2}_{quadrature}+3\underbrace{\| \uuh-\uhu \|^2}_{Galerkin}
	\end{equation}
	where $\uhhu$ is the optimal solution of the fully-discrete
        OCP \eqref{eq:full_discret_gen} and $\uuh$ is the optimal
        control of the semi-discrete OCP \eqref{eq:FE_OCP}.  The
        result is straightforward using the bounds in Corollary
        \ref{co:SAGA} and Assumptions \ref{ass:FErate} and
        \ref{ass:conv_prob}.
	\end{proof}
        Similarly, we can derive a complexity result in the case of a
        sub-exponential decay of the quadrature error, as for the
        problem in Section \ref{sec2}. The computational work model
        that we consider here is the same as in subsection
        \ref{sec:complexity_SG}.
	\begin{corollary}\label{corr:complexity2}
        Let the same hypotheses as in Theorem \ref{3T} hold, and
        assume further a quadrature error of the form
        $\theta(n)=E_4e^{-s\sqrt[M]{n}}$ for some $s>0$ and
        $M\in\mathbb{N}_+$. In order to guarantee a MSE $\E[\|
          \uS - \uu \|^2] \lesssim tol^2$, the total
        required computational work using the SAGA-IS algorithm
        is bounded by
	\begin{equation}
            W  \lesssim \left(\log(tol^{-1})\right)^{M+1} tol^{-\frac{d \gamma}{r+1}}.
	\end{equation}

        Moreover, the memory space required to store the history of
        the computed gradients at each iteration scales as
	\begin{equation}
	storage=O\left(\left(\log(tol^{-1})\right)^{M} tol^{-\frac{d}{r+1}}\right).
	\end{equation}
	\end{corollary}
	\begin{proof}
	Using Theorem \ref{3T}, as we want to guarantee a MSE of size $O(tol^2)$, we can equalize the three terms on the right hand side of \eqref{3t} to $tol^2$ and finally get:
	\[
          h\sim tol^{\frac{1}{r+1}}, \quad n\sim\left( \log(tol^{-1}) \right)^M, \quad k\sim\frac{2 \log(tol^{-1})}{-\log(1-\frac{1}{2n})},
        \]
	so we obtain asymptotically
	\[
          k \sim {4 n \log(tol^{-1})}\sim 4 \left( \log(tol^{-1}) \right)^M \log(tol^{-1})\sim \left(\log(tol^{-1})\right)^{M+1} \]
	Therefore, the total work $W=2C_h k$ scales asymptotically as
	\begin{equation}
	W \lesssim \left(\log(tol^{-1})\right)^{M+1} tol^{-\frac{d \gamma}{r+1}}.
	\end{equation}
	The memory space required to store the history of all the $n$
        computed gradients is proportional to $n h^{-d}$, so:
	\begin{equation}
	storage\lesssim \left(\log(tol^{-1})\right)^{M} tol^{-\frac{d}{r+1}}
	\end{equation}
	\end{proof}
	
	The computational work and storage requirements for SAGA
        stated in Corollary \ref{corr:complexity2} are reported in
        Table \ref{tab1}.
        \begin{table}
	\begin{center}
	\begin{tabular}{ r | c| c| c  }
	$O(\cdot)$& CG & SG & SAGA \\
	\hline
	W&$\left(\log(tol^{-1})\right)^{M+1} tol^{\frac{-d \gamma}{r+1}}$& $ tol^{-2-\frac{-d \gamma}{r+1}} $ &$\left(\log(tol^{-1})\right)^{M+1} tol^{\frac{-d \gamma}{r+1}}$ \\
	\hline
	storage& $ tol^{\frac{-d}{r+1}}$& $ tol^{\frac{-d}{r+1}}$ &$\left(\log(tol^{-1})\right)^{M} tol^{\frac{-d}{r+1}}$\end{tabular}
	\end{center} 
	\caption{Computational work and required memory storage for the SAGA-IS algorithm \eqref{SAGAref}, the SG-IS algorithm \eqref{modifiedSG} and the CG method  to solve the OCP \eqref{eq:weak_OCP}.} \label{tab1}
	\end{table}
        For comparison, we state in the same Table also the
        computational work and storage requirement of the SG-IS
        algorithm \eqref{modifiedSG}, as well as a Conjugate Gradient
        (CG) algorithm, both applied to the fully discrete OCP
        \eqref{eq:full_discret_gen} based on the same quadrature
        formula and spatial approximation as for SAGA (we refer to
        \cite{MartinKrumscheidNobile} where these results have been
        derived in the context of a Monte Carlo approximation). A
        naive implementation of the CG algorithm would require to
        store the gradient computed in each quadrature point, hence a
        storage of $O(nh^{-d})$. Alternatively, one can store only the
        partial weighted sum of the gradients and update it as soon as
        the gradient in a new quadrature point has been computed,
        which brings down the storage to $O(h^{-d})$.
	
}

{
	\section{Numerical example: elliptic PDE with transport term}\label{sec5}
	\subsection{Problem setting}
	In this section we verify the assertions on the order of
        convergence and computational complexity stated in Theorem
        \ref{3T} and Corollary \ref{corr:complexity2}. For this
        purpose, we consider the following problem, adapted from
        \cite{KouriSurowiec2018}, which fits the general framework of
        Section \ref{sec4}. The problem models the transport of the
        contaminant with the advection-diffusion equation. The goal is
        to determine optimal injection of chemicals to contrast the
        contaminant and thus minimize its total
        concentration. Uncertainties arise in the transport field $V$
        (i.e., wind), in the diffusion coefficient $\epsilon$ and in
        the sources $f$. Let \(D=[0,1]^{2}\) denote the physical
        domain and \(Y=H^{1}_{\Gamma_d}(D)\) the space of contaminant
        concentrations, vanishing on the boundary portion
        $\Gamma_d=\{0\}\times (0,1)$. The optimization problem reads
	$$
	\min _{u \in U} J(u) := \frac{1}{2}\E\left[ \int_{D} y_\omega(u)^{2} \mathrm{d} x\right]+\frac{\beta}{2} \| u \|_{L^2(D)}^2
	$$
	where $y_\omega(u)=y\in Y$ solves the variational problem
	\begin{equation}
	\label{num_primal}
	\int_{D}\{\epsilon(\omega) \nabla y \cdot \nabla v+(V(\omega) \cdot \nabla y) v\} \mathrm{d} x=\int_{D}(f(\omega)- u) v \mathrm{d} x, \quad \forall v\in Y, \;\; \text{and a.e. } \omega\in\Gamma
	\end{equation}
	Homogeneous Dirichlet boundary conditions have been applied on
        $\Gamma_d$, whereas homogeneous Neumann conditions have been
        applied on the remaining portion of the boundary. The control
        space is \(U=L^2(D)\) and $\beta$ is the regularization
        parameter, also called price of energy. The PDE coefficients
        \(\epsilon, \mathrm{V},\) and \(f\) are random fields: the
        source term is given by
        $$
	f(x,\omega)= \exp \left(-\frac{\|x-z(\omega)\|^2}{2 \sigma^{2}}\right)
	$$
	where the location $z(\omega)=(\xi_1(\omega), \xi_2(\omega))$ is a random variable uniformly distributed in $D$, and \(\sigma=1\); the diffusion coefficient is given by
	$$
	\epsilon(\omega)=0.5+ \exp (\xi_3(\omega)-1).
	$$
	with $\xi_3(\omega)$ uniformly distributed in $[0,1]$; finally, the transport field $V$ is given by
	$$
	V(x, \omega)=\left[\begin{array}{c}
	\xi_4(\omega)-\xi_5(\omega) x_{1} \\
	\xi_5(\omega) x_{2}
	\end{array}\right]
	$$
        with $\xi_4(\omega)$ and $\xi_5(\omega)$ uniformly distributed
        in $[0,1]$. All random variables $\xi_1,\ldots,\xi_5$ are
        mutually independent. The weak formulation of the adjoint
        problem simply writes: find $p=p_\omega(u)\in Y$ such that
	\begin{equation}
	\int_{D}\{\epsilon(\omega) \nabla v \cdot \nabla p+ (V(\omega) \cdot \nabla v) p\} \mathrm{d} x=\int_{D}y v \mathrm{d} x, \quad \forall v\in Y, \;\; \text{and a.e. } \omega\in\Gamma
	\end{equation}
        and has homogeneous Dirichlet boundary conditions on
        $\Gamma_d$ and homogeneous co-normal derivative
        $\frac{\partial p(u)}{\partial n_{A}}=\epsilon\frac{\partial
          p}{\partial n} + (V\cdot n) p = 0$ on the remaining portion
        of the boundary. The gradient of $J$ then writes
	$$\nabla J(u)=\beta u + \E[p_\omega(u)].$$ We have chosen
        $\beta=10^{-4}$ in the objective functional. For the FE
        approximation, we have considered a structured triangular grid
        of mesh size $h$ where each side of the domain $D$ is divided
        into $1/h$ sub-intervals and used piece-wise linear finite
        elements (i.e. $r=1$). For the approximation of the
        expectation in the objective functional, we have used a full
        tensor Gauss-Legendre quadrature formula with the same number
        $q$ of quadrature knots in each random variable $\xi_j$,
        $j=1,\ldots,5$. All calculations have been performed using the
        FE library Freefem++ \cite{MR3043640}.

        \subsection{Performance of SAGA and comparison with CG}
	In this subsection, we consider the SAGA method using a fixed
        mesh size $h=2^{-3}$ and study its convergence, with respect
        to the iteration counter, for different levels of the full
        tensor Gauss-Legendre quadrature formula, i.e. a different
        number $q$ of points in each random variable (the total number
        of quadrature points being $q^5$). For each $q$, we compute a
        reference solution $u_{ref}$ by CG up to convergence (using
        the same FE mesh size $h=2^{-3}$). Then, we perform $20000$
        SAGA iterations and compute the $error=u_k-u_{ref}$, at
        each iteration $k$, w.r.t. the reference solution.  We repeat
        the computation $40$ times, independently, to estimate the
        log-mean error $\log \E[\|error\|]$ (hereafter $\log(\cdot)$
        refers to the base 10 logarithm). In all cases we have used a
        step size $\tau=10$.
	\begin{figure}
	\begin{subfigure}{.49\textwidth}
	  \centering
	  \includegraphics[width=1\linewidth]{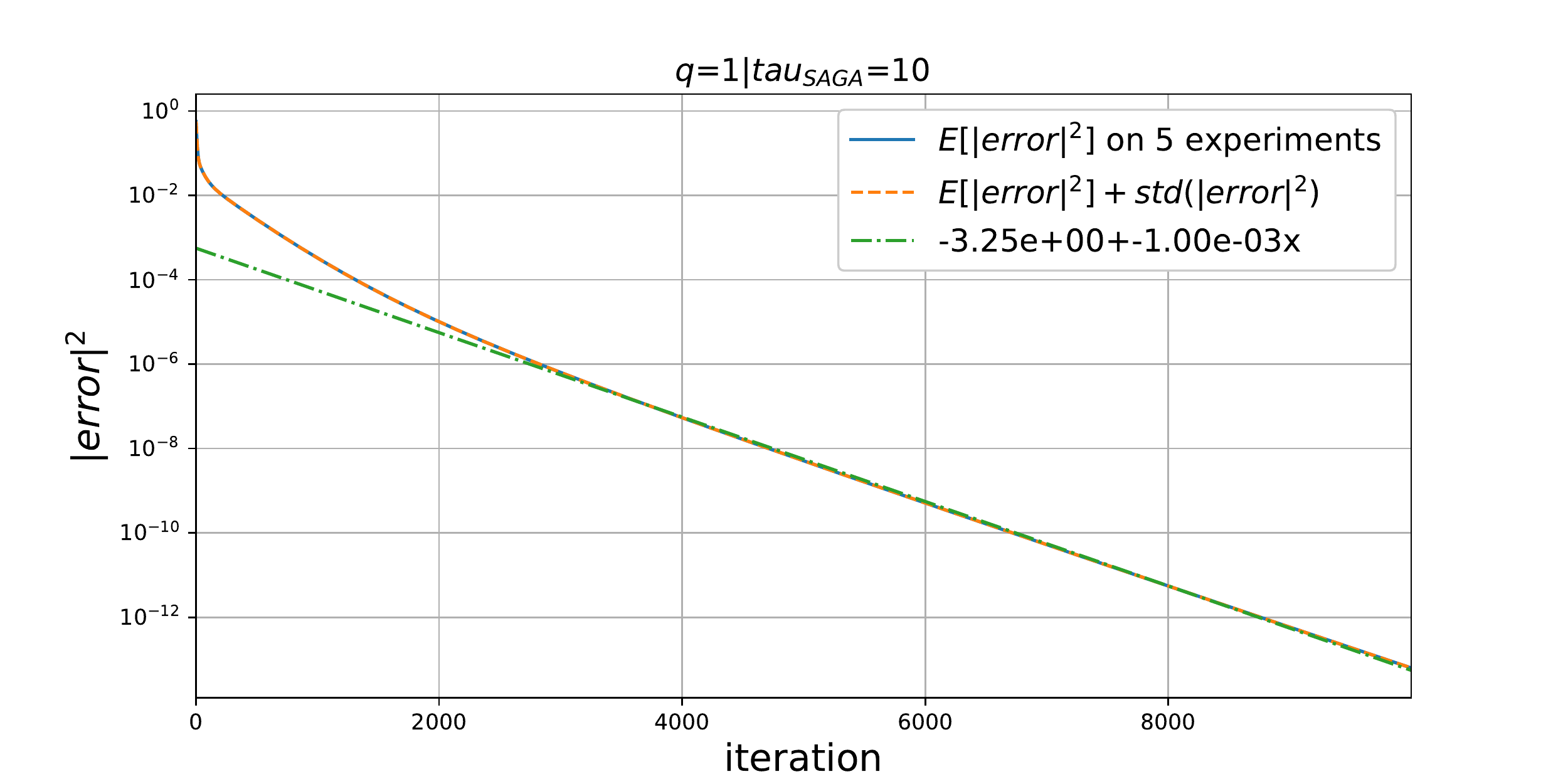}  
	  \caption{$q=1$}
	  \label{fig:sub-00}
	\end{subfigure}
	\begin{subfigure}{.49\textwidth}
	  \centering
	  \includegraphics[width=1\linewidth]{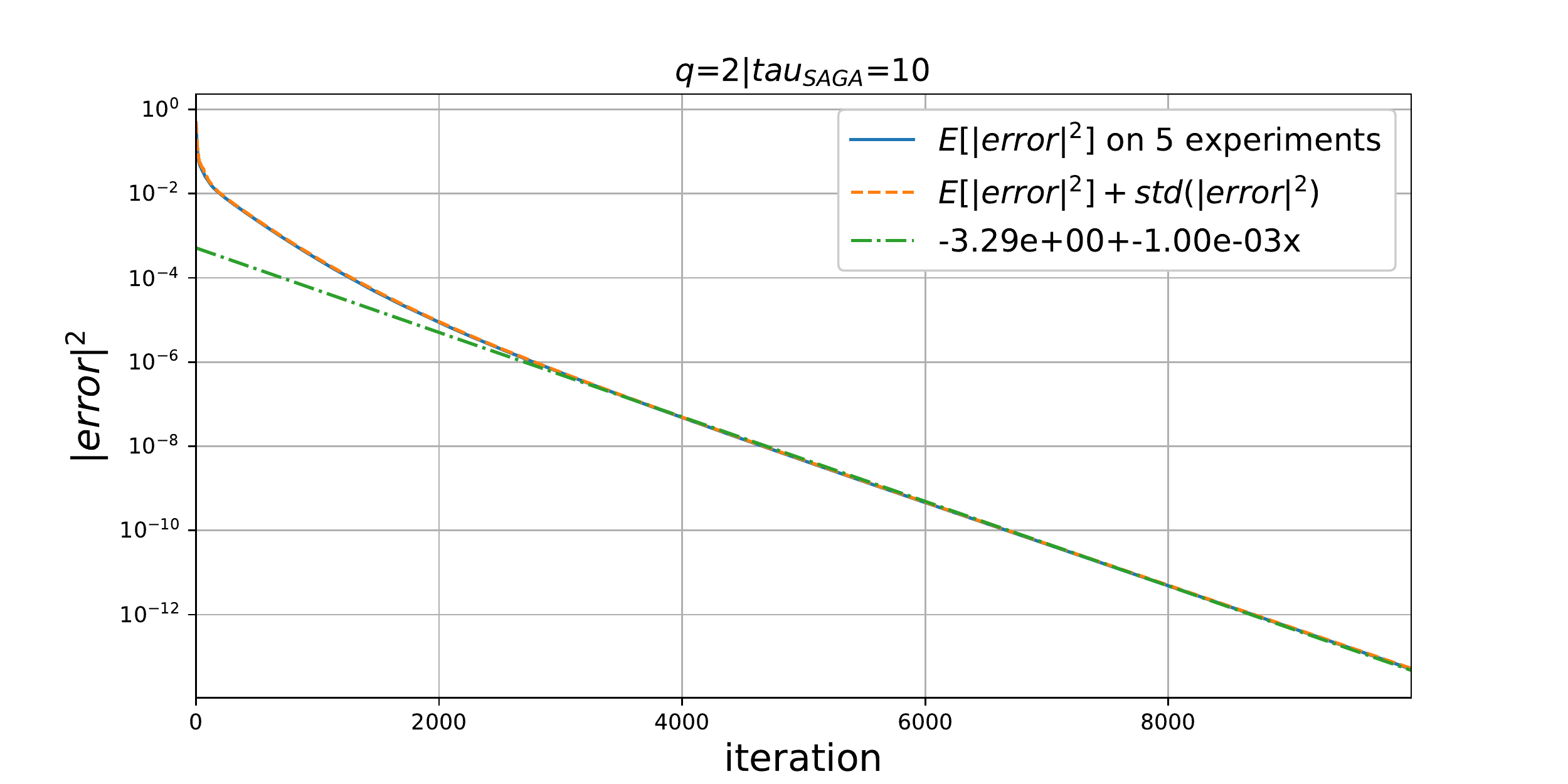}  
	  \caption{$q=2$}
	  \label{fig:sub-0}
	\end{subfigure}
	\newline
	\begin{subfigure}{.49\textwidth}
	  \centering
	  \includegraphics[width=1\linewidth]{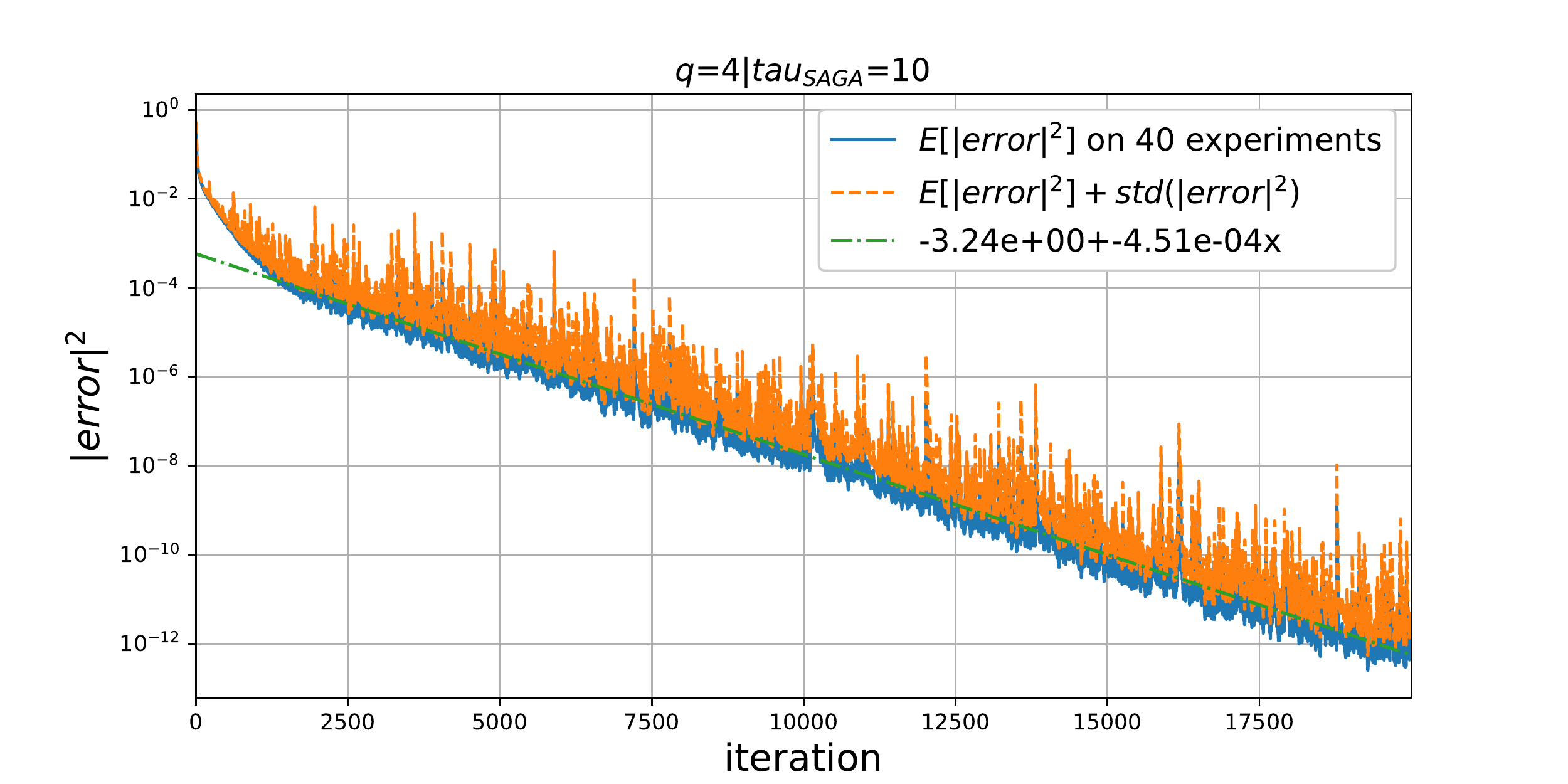}  
	  \caption{$q=4$}
	  \label{fig:sub-first}
	\end{subfigure}
	\begin{subfigure}{.49\textwidth}
	  \centering
	  \includegraphics[width=1\linewidth]{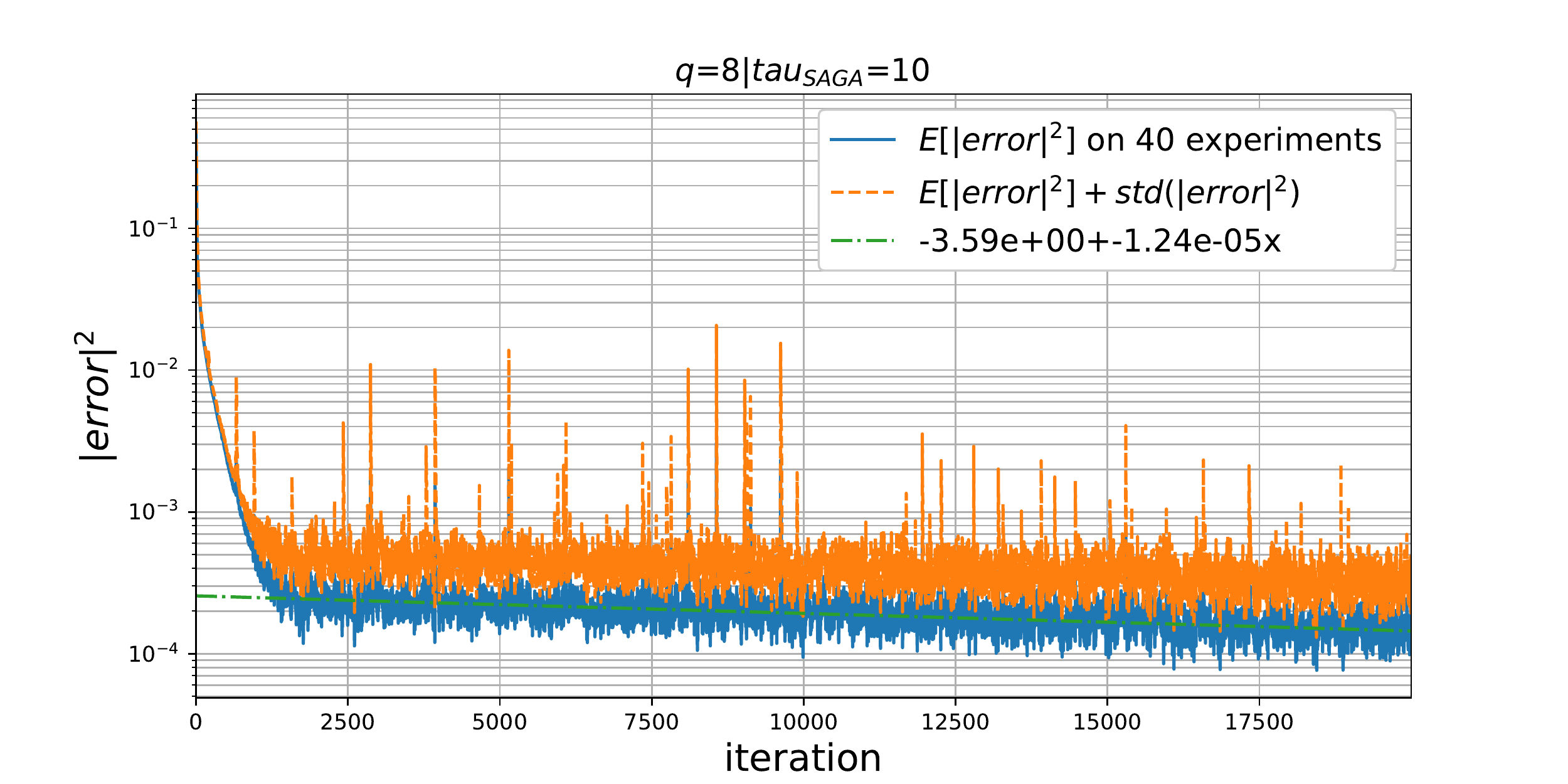}  
	  \caption{$q=8$}
	  \label{fig:sub-second}
	\end{subfigure}
	\newline
	\begin{subfigure}{.49\textwidth}
	  \centering
	  \includegraphics[width=1\linewidth]{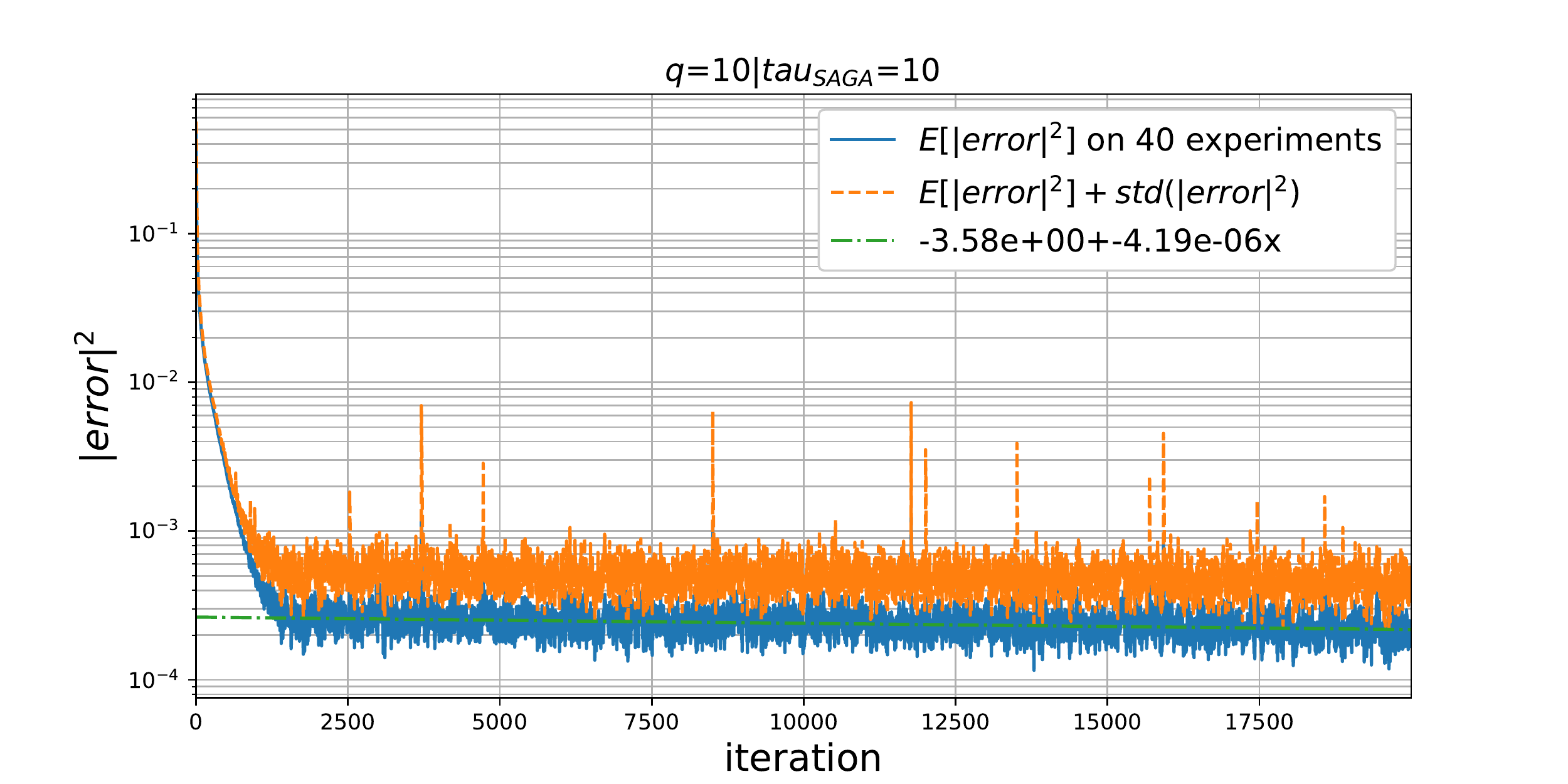}  
	  \caption{$q=10$}
	  \label{fig:subthird}
	\end{subfigure}
	\begin{subfigure}{.49\textwidth}
	  \centering
	  \includegraphics[width=1\linewidth]{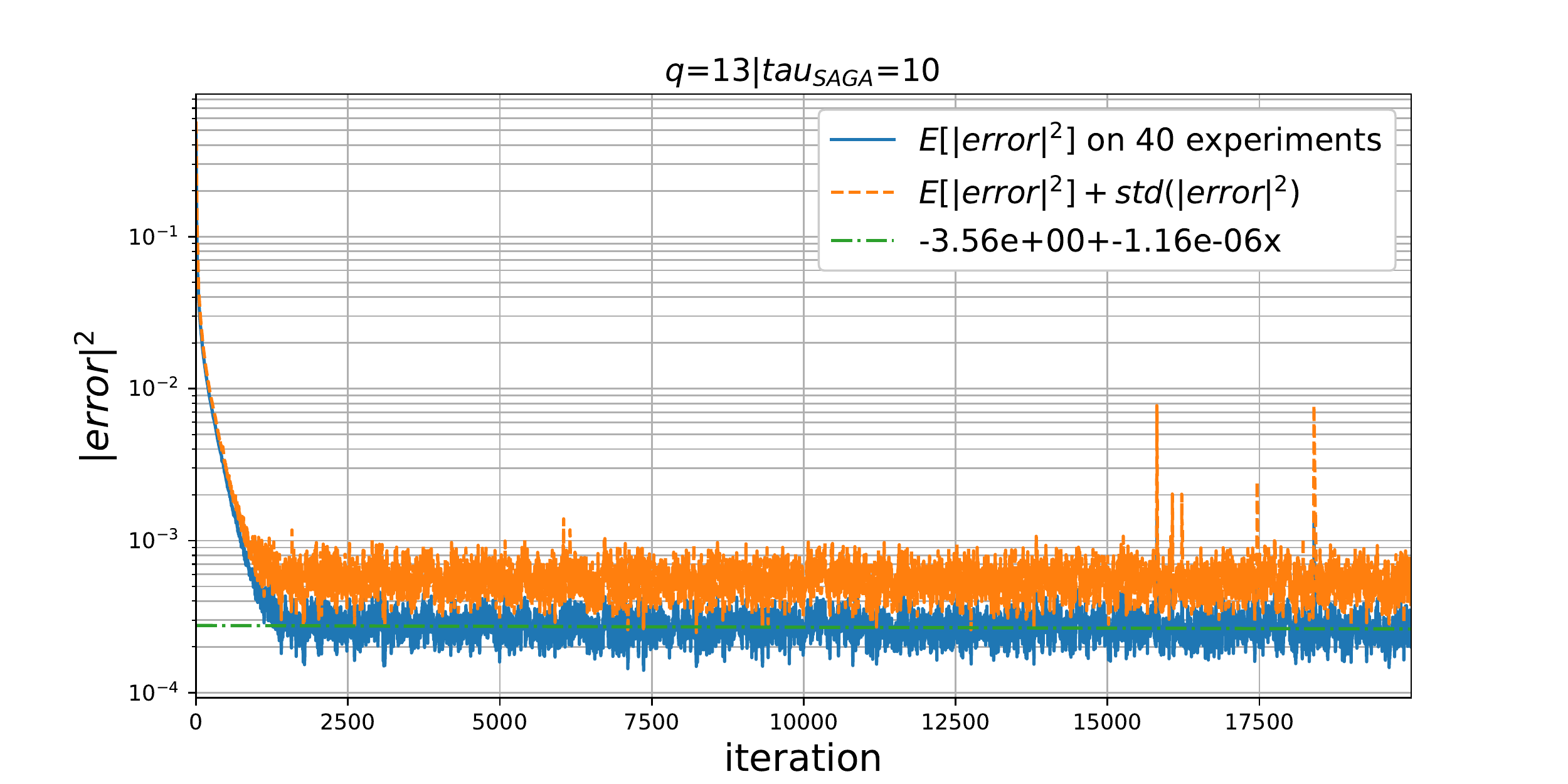}  
	  \caption{$q=13$}
	  \label{fig:sub-fourth}
	\end{subfigure}
	\caption{Convergence of SAGA w.r.t. the iteration counter for different $q$ and fixed FE mesh}
	\label{convSAGAinit0}
	\end{figure}
	
	We show in Figure \ref{convSAGAinit0} the convergence plots of
        $\log(\E[\| error \|])$ versus the number of iterations, for
        $q \in \{ 2,4,8,13 \}$ and in Figure \ref{SAGAonlyOverqZoom} a
        zoom on the first $2000$ iterations. We clearly observe two
        regimes: a first one over the first few hundreds of
        iterations, of faster exponential convergence, and a second
        one afterwards, of slower, but still exponential convergence.
        We assess hereafter that the slower rate observed behaves as
        predicted by our analysis in Corollary \ref{co:SAGA}. We are
        not able to explain at the moment, however, the faster initial
        regime.
	\begin{figure}[!ht]
	\includegraphics[width=1\textwidth]{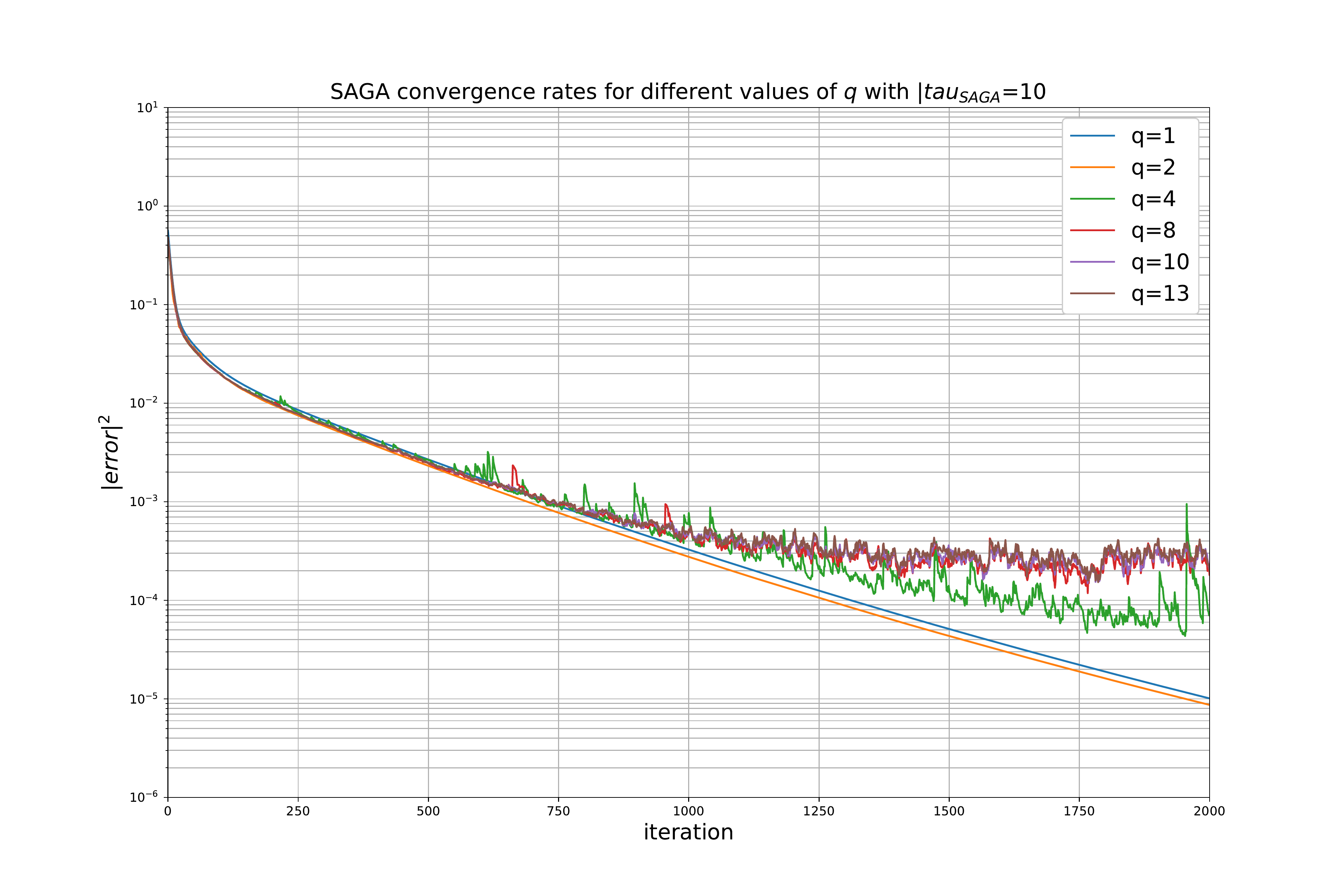}
	\caption{Convergence of SAGA algorithm w.r.t. the iteration counter, for different values of $q$ and fixed FE mesh. Zoom of Figure \ref{convSAGAinit0} on the first $2000$ iterations}
	\label{SAGAonlyOverqZoom}
	\end{figure}
	
	In order to verify quantitatively the exponential error decay
        $\ea (1-\epsilon)^k$ of equation \eqref{SAGArate}, with
        $\epsilon \sim \frac{1}{n}$, we have estimated by least squares
        fit the constant $\ea$ and rate $\epsilon$ for each value of
        $q$ {(we have taken only the second regime into consideration
        in the least squares fit)}. In Figure
        \ref{constSAGAinit0L}, we plot the ratio between the estimated
        convergence rate $\epsilon_{est}$ and the theoretical one
        $\epsilon_{th}=\frac{1}{2n}=\frac{1}{2q^5}$ for the different
        tested values of $q$. Similarly, in Figure
        \ref{constSAGAinit0R} we plot the estimated constant $\ea$ versus $q$.
          estimate the constants but you seem to be using some
          In both cases, the plotted quantities
        vary very little with $q \in \{ 3, \dots, 13 \}$ which
        confirms the validity of our theoretical analysis.
	\begin{figure}[!ht]
	\begin{subfigure}{.49\textwidth}
	  \centering
	 \resizebox{\columnwidth}{!}{ \input{SAGA_cst_est_eps.tex}  }
	  \caption{$\epsilon_{exp.}/\epsilon_{th}$ versus $q$}
	   \label{constSAGAinit0L}
	\end{subfigure}
	\begin{subfigure}{.49\textwidth}
	  \centering
	\resizebox{\columnwidth}{!}{  \input{SAGA_cst_est_C.tex}  }
	  \caption{Estimated constant $\ea$ versus $q$}
	\label{constSAGAinit0R}
	\end{subfigure}
	\caption{Estimation of the SAGA rate $\epsilon$ and constant $\ea$ versus $q$}
	\label{constSAGAinit}
	\end{figure}
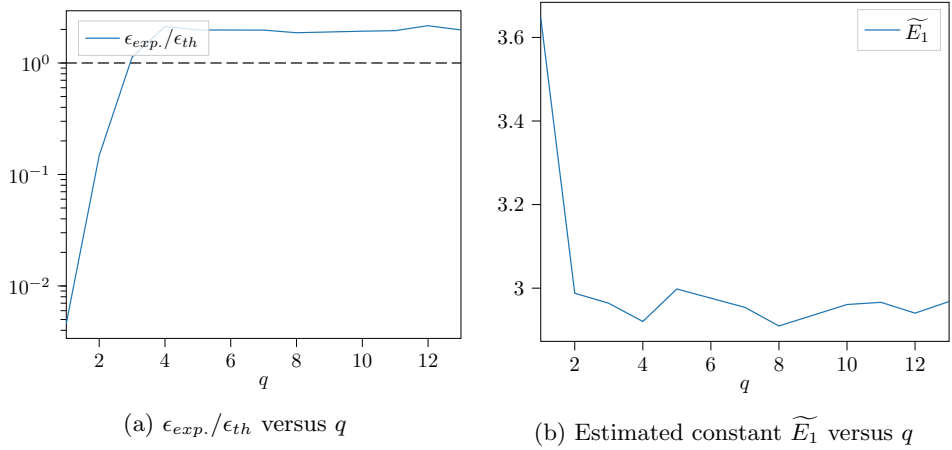

        We analyze next the sensitivity of the SAGA algorithm
        w.r.t. the step-size $\tau$. In Figure \ref{SAGAsensitivity},
       we plot the mean
          squared error $\E[\|u_k-u_{ref}\|^2]$ versus the iteration
          counter, for $q=5$, where the expectation is estimated from
          $5$ independent experiments. With a large step-size,
        $\tau=100$,
        the method diverges to infinity. Progressively decreasing
        $\tau$ to $10$, the method starts converging, with the
        expected exponential rate, although slowly. Among the
        different values that we have tried, a step size $\tau=10$
        seems to provide the fastest convergence. Further decreasing
        $\tau$ to $0.01$ makes SAGA converge poorly.
	\begin{figure}[!ht]
	\includegraphics[width=1\textwidth]{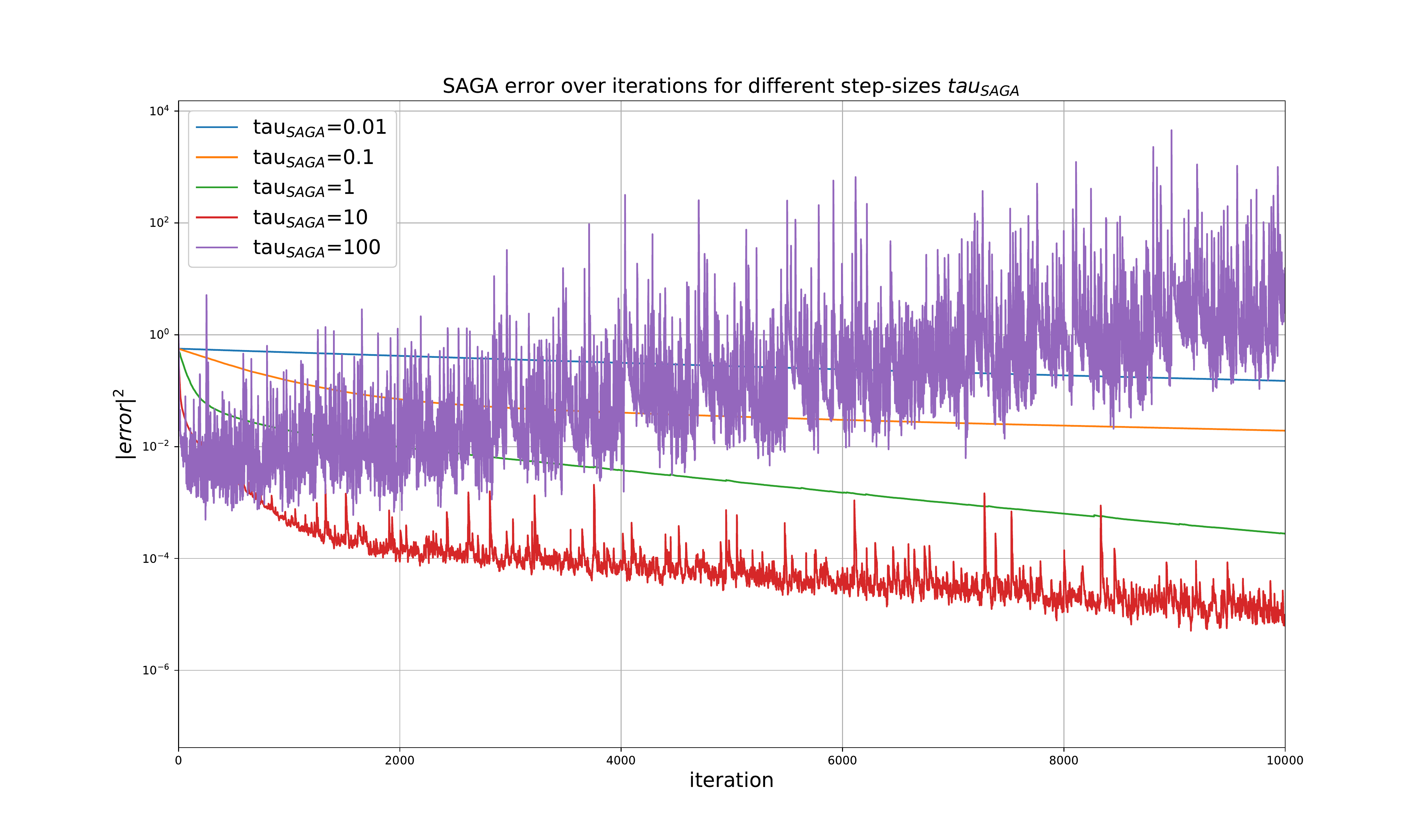}
	\vspace{-10mm}
	\caption{Sensitivity of SAGA (with $q=5$) to the choice of the step size $\tau$.}
	\label{SAGAsensitivity}
	\end{figure}
	
	In Figure \ref{SAGAvsFGbothopt1} we compare the convergence
        rate of SAGA, with $\tau=10$ with that of
        CG.  The error in both cases is plotted against the total
        number of PDE solves. The plot shows a faster convergence of
        CG than SAGA, asymptotically. However, SAGA features a smaller
        error in the pre-asymptotic regime, and delivers an acceptable
        solution, from a practical point of view, already before two
        full iterations of CG (2500 PDE solves). This makes it
        attractive in a limited budget context.
	\begin{figure}[!ht]
	\includegraphics[width=1\textwidth]{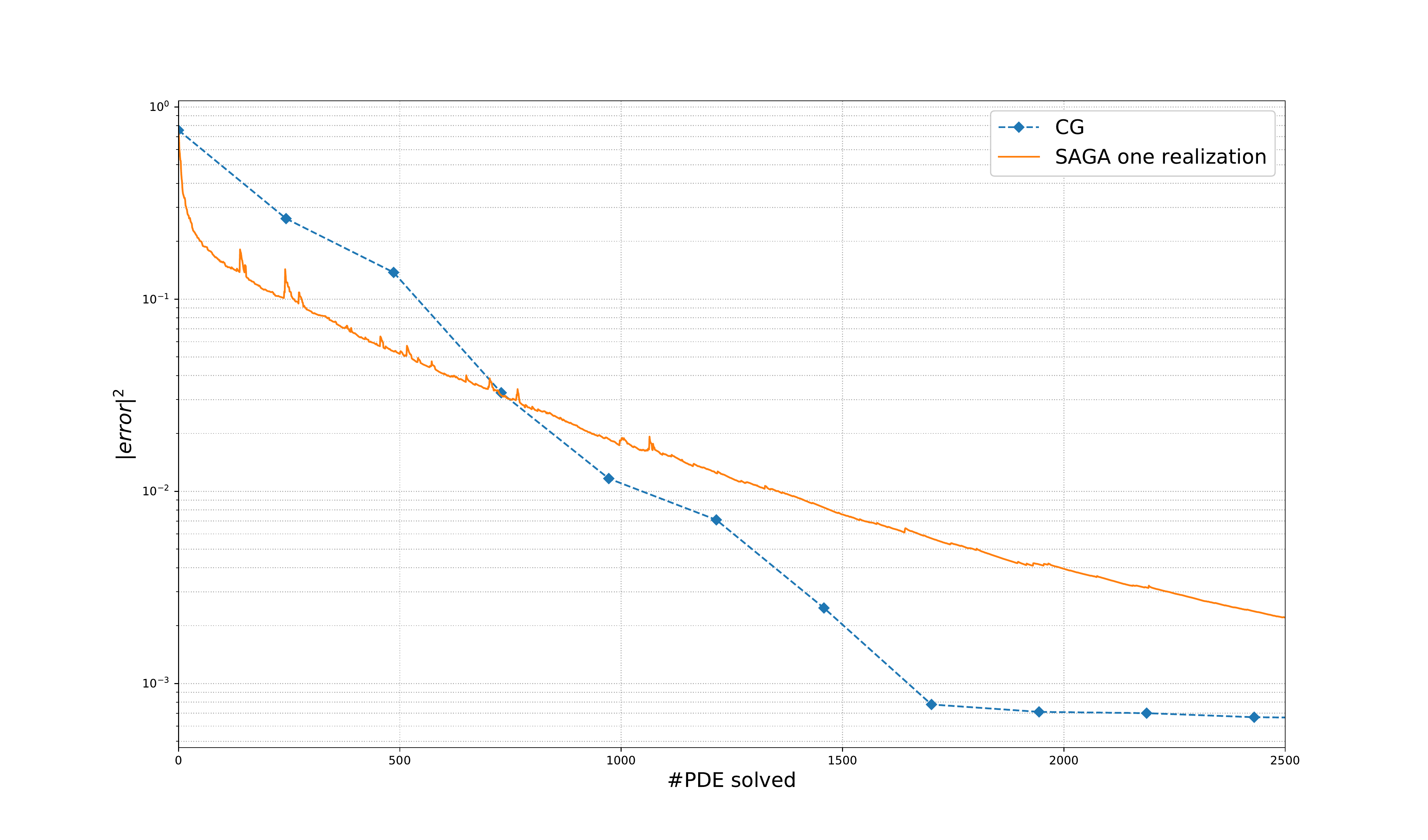}
	\vspace{-10mm}
	\caption{Comparison of SAGA with $\tau=10$ vs CG, for $q=3$, $h=\frac{1}{55}$.}
	\label{SAGAvsFGbothopt1}
	\end{figure}

	\subsubsection{Complexity results for the SAGA algorithm}
	We investigate here the complexity of the SAGA method \eqref{SAGAref}, for which we recall the error bound \eqref{3t} in the
	case of piece-wise linear FE (i.e.\ $r=1$) and a 5-dimensional probability space (i.e. $M=5$):
	\begin{equation}\label{3tnew}
	\E[\| \uS-\uu \|^2]\leq \ea (1-\epsilon(q))^k+\eb e^{-s q}+\ec h^{4}
	\end{equation}
	where $s$ is the rate of exponential convergence of the quadrature formula, and $q$ the number of knots used in each stochastic variable.
        To assess the complexity of the method and balance the error
        contributions, we need first to estimate the constants $\ea$,
        $\eb$, $\ec$ and rates $\epsilon$ and $s$. For this, we have
        proceeded in the following way:
        
       	\begin{itemize}
	\item \emph{Estimation of $\ea$ and $\epsilon=\epsilon(q)$.}
          This has been done already in the previous section using a
          fixed mesh with $h=2^{-3}$. From Figure \ref{constSAGAinit}
          we infer the values $\ea \approx 2.95$ and $\epsilon(q)
          \approx \frac{1.0}{q^5}$.  
\item \emph{Estimation of $\eb$ and $s$.} We used again a
          fixed mesh of size $h=2^{-3}$. First, we computed the
          reference solution with a fine Gauss-Legendre quadrature
          formula, i.e. $q=8$, using the CG algorithm until
          convergence. Then we computed the error for the approximated
          optimal control using only $q \in \{ 1, \dots, 7 \}$ points
          in each random variable, using again the CG algorithm until convergence.
            Results are detailed in Table \ref{c2}
          and plotted in Figure \ref{ColEst}. The error is the
          difference between the estimated optimal control using $q
          \in \{ 1, \dots, 7 \}$ knots in the quadrature formula, and
          the optimal control computed for $q=8$. We estimate $\eb
          \approx 10.43$ and $s \approx 3.697$.
	\begin{table}
	\begin{center}
	\begin{tabular}{c c}
	\hline
	  $q$  &        $\|error\|^2$ \\
	\hline
	 1 &  3.501974e-03 \\
	 2 &  7.842113e-07 \\
	 3 &  7.583597e-11 \\
	 4 &  6.019157e-15 \\
	 5 &  1.281663e-18 \\
	 6 &  3.673762e-22 \\
	 7 &  7.141313e-25 \\
	\hline
	\end{tabular}
	\caption{Quadrature error on the optimal control, versus the number of knots $q$ used in each random variable.} \label{c2}
	\end{center}
	\end{table}
	
	\begin{figure}[!ht]
	\includegraphics[width=1\textwidth]{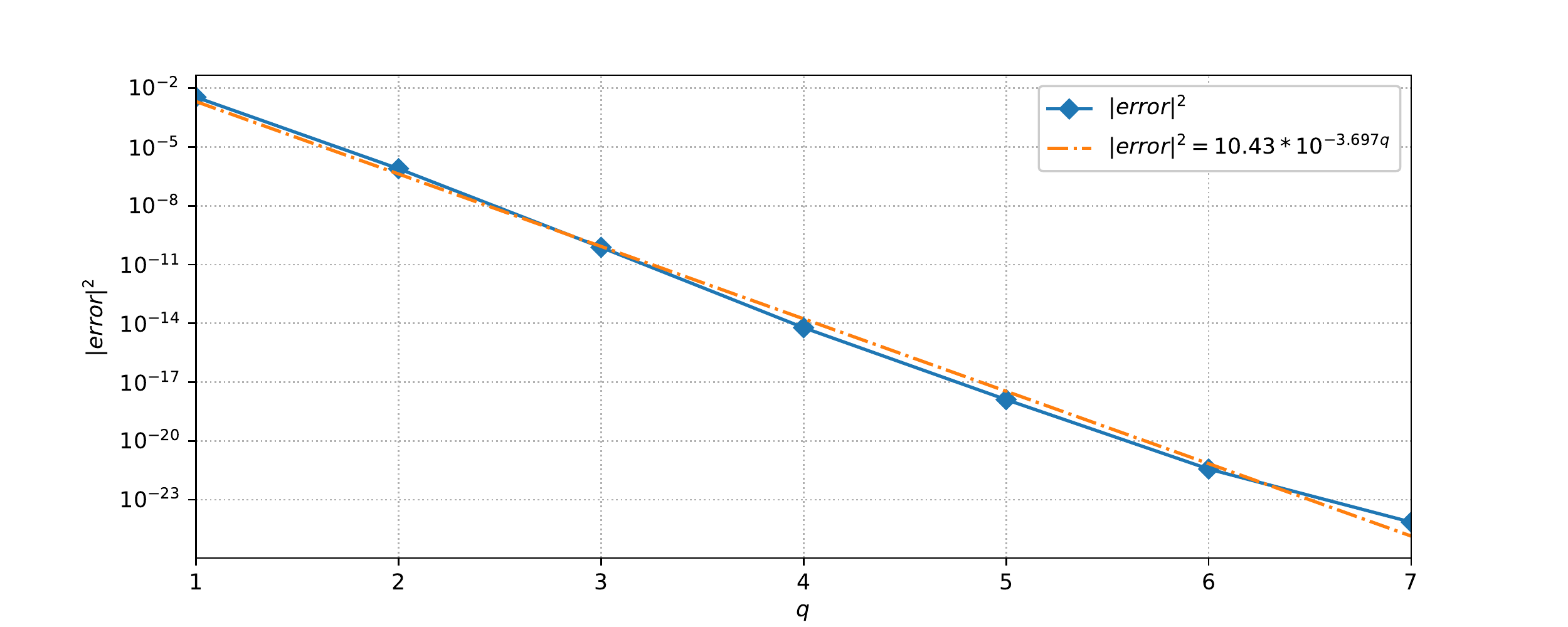}
	\caption{Fitting the (squared) quadrature error model $\eb e^{-s q}$ in \eqref{3tnew}.}
	\label{ColEst}
	\end{figure}

	\item \emph{Estimation of $\ec$.} We used a quadrature formula
          with $q=1$ knot in each random variable, and computed the
          reference solution on a fine mesh (with $h=2^{-8}$) using
          the CG algorithm until convergence. Then we
          run again the CG algorithm with the same setting,
          yet on coarser meshes with size $h=2^{-1}, \dots,
          2^{-8}$. The results are shown in Table \ref{c3FG} and Figure
          \ref{FEEst} which confirm a convergence of order $h^4$ of the squared error with an estimated constant $\ec \approx 2.96$.
	\begin{table}
	\begin{center}
	\begin{tabular}{c c}
	\hline
	    $h^{-1}$ &        $\|error\|^2$ \\
	\hline
	   2 &  1.153577e-01 \\
	   4 &  1.544431e-02 \\
	   8 &  1.068433e-03 \\
	  16 &  7.209996e-05 \\
	  32 &  4.561293e-06 \\
	  64 &  2.824879e-07 \\
	 128 &  1.695013e-08 \\
	 256 &  7.707942e-10 \\
	\hline
	\end{tabular}
	\end{center} 
	\caption{FE error on the optimal control (computed with $q=1$) versus the mesh size $h$.} \label{c3FG}
	\end{table}

	\begin{figure}
	\includegraphics[width=1\textwidth]{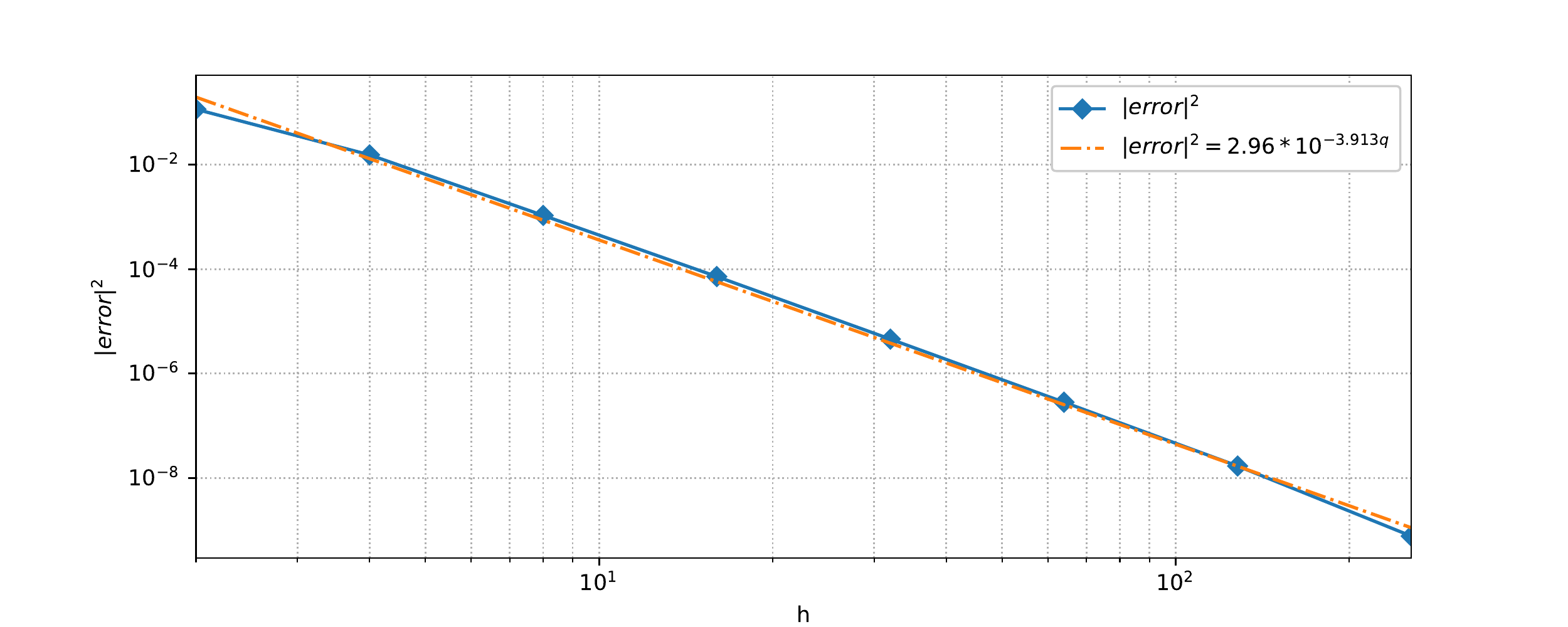}
	\caption{Fitting the (squared) FE error model $\ec h^{-4}$ in \eqref{3tnew}.}
	\label{FEEst}
	\end{figure}
	\end{itemize}

        We are now ready to assess the complexity of the SAGA
        algorithm. For this, for a given tolerance $tol$, we compute
        the optimal mesh size $h(tol)$, the optimal number of
        Gauss-Legendre points in the quadrature formula $q(tol)$, and
        the optimal number of SAGA iterations $k_{max}(tol)$, using
        the constants $\ea$, $\eb$ and $\ec$ and rates $\epsilon(q)$
        and $s$ estimated above. Then we run the SAGA algorithm up to
        iteration $k_{max}$, on a mesh of size $h(tol)$, using a
        quadrature formula with $q(tol)$ points in each random
        variable, and plot the  error on the optimal control for a single realization.
        The reference solution has been computed
        using the CG method, on a mesh of size $h=2^{-8}$, with a
        quadrature formula with $q=5$, until we reach a very small tolerance, ($k_{max}=30$, $\| \nabla J(u_{30})\| \leq 10^{-17}$, $\| u_{30}-u_{29} \| \leq 10^{-13}$)
        Figure \ref{fullWSAGA} shows the errors of such SAGA
        simulations versus the computational cost model $W=k_{max}
        h^{-d}$. Table \ref{fullSaga} gives details on
        the optimal discretization parameters $h, q, k_{max}$ as well
        as the required complexity for each considered tolerance.
        \begin{figure}[!ht]
	\includegraphics[width=1\textwidth]{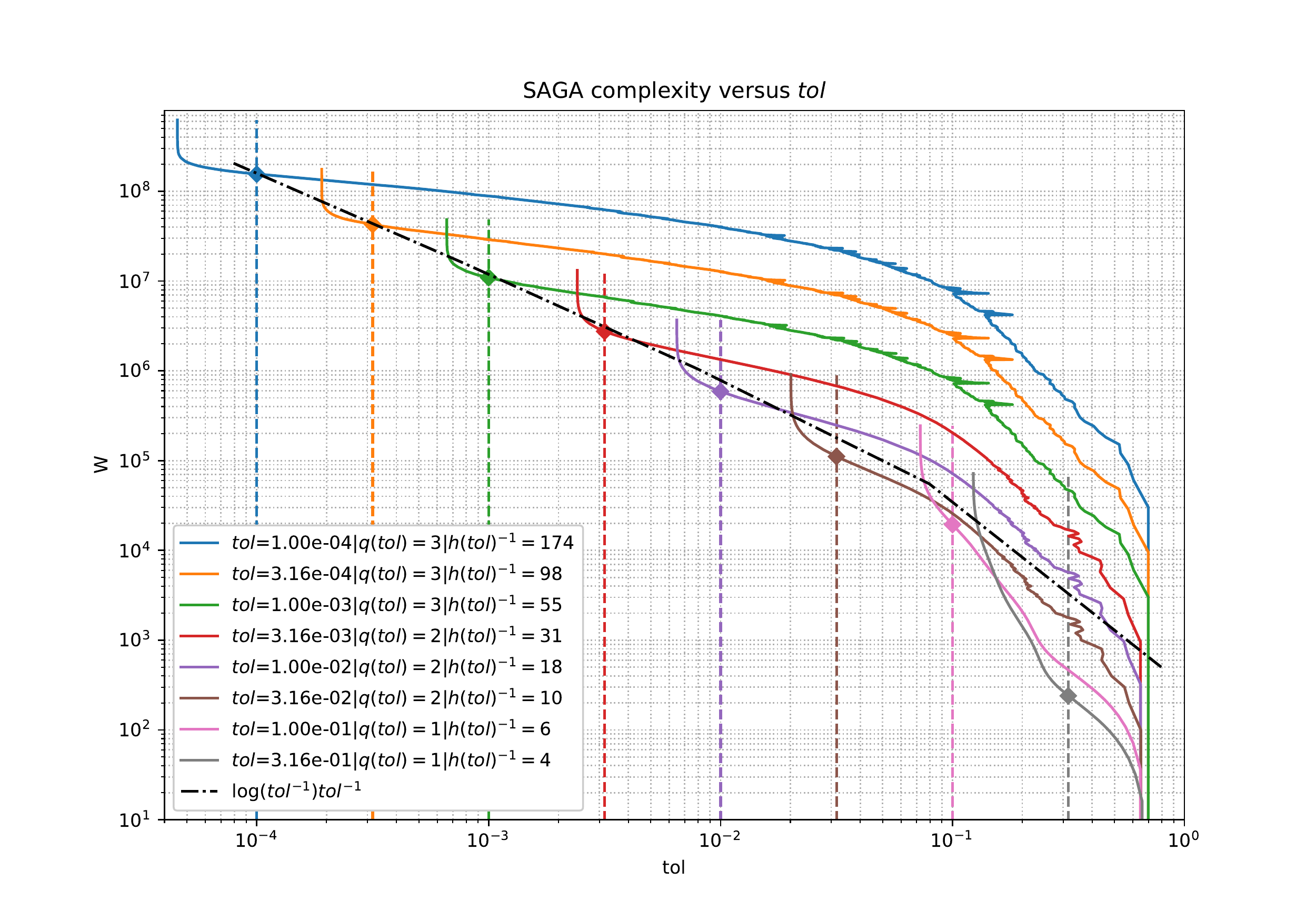}
	\vspace{-5mm}
	\caption{Computational work model versus estimated error  for the SAGA method \eqref{SAGAref}. Each curve represents the evolution of the error along the SAGA iterations; the diamond corresponds to the iteration $k_{\text{fst}}$, at which the targeted tolerance is first met. The optimal (theoretical) iteration counter $k_{max}(tol)$ is given for comparison purpose in Table \ref{fullSaga}.}
	\label{fullWSAGA}
	\end{figure}
	\begin{table}
	\begin{center}
	\begin{tabular}{cccccccc}
	\hline
	       $tol$ &  $q(tol)$ &   $h(tol)^{-1}$ & $k_{max}(tol)$ & error$(k_{max}(tol))$ & $k_{\text{fst}}$ & error$({k_{\text{fst}}})$ & $W$ \\
	\hline
	 3.162e-01 &  1 &    4 &  1124 &  1.237e-01 &        15 &      3.156e-01 &     2.400e+02 \\
	 1.000e-01 &  1 &    6 &  1699 &  7.277e-02 &       539 &      9.997e-02 &     1.940e+04 \\
	 3.162e-02 &  2 &   10 &  2274 &  2.017e-02 &      1111 &      3.161e-02 &     1.111e+05 \\
	 1.000e-02 &  2 &   18 &  2849 &  6.495e-03 &      1818 &      9.997e-03 &     5.890e+05 \\
	 3.162e-03 &  2 &   31 &  3424 &  2.415e-03 &      2864 &      3.162e-03 &     2.752e+06 \\
	 1.000e-03 &  3 &   55 &  3999 &  6.611e-04 &      3620 &      9.995e-04 &     1.095e+07 \\
	 3.162e-04 &  3 &   98 &  4574 &  1.909e-04 &      4443 &      3.160e-04 &     4.267e+07 \\
	 1.000e-04 &  3 &  174 &  5149 &  4.556e-05 &      5125 &      9.996e-05 &     1.552e+08 \\
	\hline
	\end{tabular}
	\end{center} 
	\caption{Final error for one realization of SAGA, versus the target tolerance $tol$. $k_{\text{fst}}$ corresponds to the iteration at which the targeted tolerance is first met.} \label{fullSaga}
	\end{table}
	The results follow well the predicted theoretical complexity
        given in Corollary \ref{corr:complexity2}, at least for small tolerances.

	\section{Conclusions}\label{conclusions}
	In this work, we have proposed a SAGA algorithm with
        importance sampling to solve numerically a quadratic optimal
        control problem for a coercive PDE with random coefficients,
        minimizing and expected regularized mean squares loss
        function, where the expectation in the objective functional
        has been approximated by a quadrature formula whereas the PDE
        has been discretized by a Galerkin method. The SAGA algorithm
        is a Stochastic Gradient type algorithm with a fixed-length
        memory term, which computes at each iteration the gradient of
        the objective functional in only one quadrature point,
        randomly chosen from a possibly non-uniform distribution. We
        have shown that the asymptotically optimal sampling
        distribution is the uniform one, over the quadrature points.
        We have also shown that, when equilibrating the three sources
        of errors, namely the PDE discretization error, the quadrature
        error and the error due to the SAGA optimization algorithm,
        the overall complexity, in terms of computational work versus
        prescribed tolerance, is asymptotically the same as the one of
        gradient based methods, such as Conjugate gradient
        (i.e. methods that sweep over all quadrature points at each
        iteration), as the tolerance goes to zero. However, as
        illustrated by our numerical experiments, the advantage of
        SAGA with respect to CG is in the pre-asymptotic regime, as
        acceptable solutions may be obtained already before a full
        sweep over all quadrature points is completed. The main
        limitation we see for this method is the required memory
        space, since the storage increases as the desired tolerance
        gets smaller, hence the number of quadrature points increases,
        and will reach at some point the memory limit of the employed
        machine. Although we point out that only one ``gradient'' term
        has to be updated at each iteration. Therefore, one could also
        dump all gradients on the hard disk and overwrite only one
        term per iteration. This will slightly slow down the
        execution, but performance will still be acceptable as long as
        the communication/access (a.k.a. Disk I/O) time is negligible
        compared to the time to solve one PDE.	
	
	As a particular example, we have considered a diffusion
        equation with diffusivity coefficient that depends
        analytically on a few uniformly distributed and independent
        random variables. For this problem, we have considered and
        analyzed a full tensor Gauss-Legendre quadrature formula,
        which, however, is affected by the curse of dimensionality,
        hence applicable only to problems for which the randomness can
        be described in terms of a small number of random
        variables. To overcome such curse of dimensionality, one could
        use sparse quadratures instead \cite{NobileSto, LeeGunzburger,
          Borzi2010}, whose weights, however, are not all
        positive. The result in Lemma \ref{SAGArecu} is still valid,
        as long as the approximate functional $\widehat{J}$ satisfies
        a strong convexity condition,
	\[
	  \frac{l}{2} \| u_1-u_2 \|^2 \leq \langle \nabla
          \widehat{J}(u_1)-\nabla \widehat{J}(u_2) , u_1-u_2 \rangle
          \quad \forall u_1,u_2 \in U,
	\]
	which might hold only for a sufficiently large number of
        quadrature points (see Remark
        \ref{rem:nonpositive_weights}). Also, because of the presence
        of negative weights, the quantity $\Stt$ might not be
        uniformly bounded in $n$, therefore, the results in Theorem
        \ref{generalRecu} and Corollary \ref{co:SAGA} might not apply
        to this case. These issues will be further investigated in a
        future work.  }
        The complexity analysis of the SAGA method presented in
          Section \ref{sec:complexity} assumed a sub-exponential decay
          of the quadrature error. Analogous complexity results could
          be derived under the assumption of an algebraic decay rate
          of the quadrature error as, for instance, for a MC or QMC
          quadrature. We should point out, however, that for a MC
          quadrature formula, the overall complexity will be dominated
          by the slow $n^{-\frac{1}{2}}$ Monte Carlo convergence rate
          and SAGA will not improve the complexity of SG, and is
          therefore not recommended in this case. On the other hand,
          SAGA gains over SG whenever the quadrature formula has a
          better convergence rate than MC. This will be the case for a
          QMC quadrature formula.

        \section*{Acknowledgments}
	F. Nobile and M. C. Martin have received support from the
        Center for ADvanced MOdeling Science (CADMOS). The second
        author acknowledges the support of the Swiss National Science
        Foundation under the Project n. 172678 “Uncertainty
        Quantification techniques for PDE constrained optimization and
        random evolution equations”.
	
	\bibliographystyle{siamplain}
	\bibliography{bib}
	\end{document}

%% file: ex_shared1.tex
\usepackage{lipsum}
\usepackage{amsfonts}
\usepackage{amsmath}
\usepackage{amssymb}
\usepackage{graphicx}
\usepackage{epstopdf}
\usepackage{algorithmic}
\usepackage{tikz}
\usepackage{epsfig}

\usepackage{support-caption}
\usepackage{subcaption}
\usepackage{caption}

\usepackage{pgfplots}
\ifpdf
  \DeclareGraphicsExtensions{.eps,.pdf,.png,.jpg}
\else
  \DeclareGraphicsExtensions{.eps}
\fi
\usepackage{hyperref}
\usepackage{breqn}
\newtheorem{assumption}{Assumption}

\newcommand{\R}{\mathbb{R}}
\newcommand{\Prob}{\mathbb{P}}
\newcommand{\N}{\mathbb{N}}
\newcommand{\E}{\mathbb{E}}

\newcommand{\gtomega}{\widetilde{g}_\omega}

\newcommand{\ii}{\mathbf{i}}
\newcommand{\qq}{\mathbf{q}}

\newcommand{\uu}{u^*}
\newcommand{\uuh}{u^{h*}}
\newcommand{\uhat}{\widehat{u}}
\newcommand{\uhu}{\widehat{u}^*}
\newcommand{\uhhu}{\widehat{u}^{h*}}
\newcommand{\uS}{\widehat{u}_{SAGA_k}^{h}}

\newcommand{\Stt}{\widetilde{S}_n}
\newcommand{\Sti}{\widetilde{S}}

\newcommand{\ea}{\widetilde{E_1}}
\newcommand{\eb}{\widetilde{E_2}}
\newcommand{\ec}{\widetilde{E_3}}

\newcommand{\St}{S_n}

\DeclareMathOperator*{\argmin}{arg\,min}
\DeclareMathOperator*{\dive}{div}

\newsiamremark{remark}{Remark}
\newsiamremark{hypothesis}{Hypothesis}
\crefname{hypothesis}{Hypothesis}{Hypotheses}
\newsiamthm{claim}{Claim}

\headers{{PDE}-constrained OCP with uncertain parameters using SAGA}{M.~Martin, and F.~Nobile}

\title{{PDE}-constrained optimal control problems with uncertain parameters using SAGA\thanks{Submitted to the editors on October 31st, 2018.
\funding{F. Nobile and M. Martin have received support from the
Center for ADvanced MOdeling Science (CADMOS).}}}

\author{Matthieu~C.~Martin\thanks{Criteo AI Lab - Grenoble, 4 Rue des Méridiens, 38130 Échirolles France
  (\email{mat.martin@criteo.com}).}
\and Fabio~Nobile\thanks{CSQI, Institute of Mathematics,
  {\'E}cole Polytechnique F{\'e}d{\'e}rale de Lausanne, 1015 Lausanne,
  Switzerland
(\email{fabio.nobile@epfl.ch}).}
  }

\usepackage{amsopn}


%% file: SAGA_cst_est_eps.tex
\begin{tikzpicture}

\definecolor{color0}{rgb}{0.12156862745098,0.466666666666667,0.705882352941177}

\begin{axis}[
legend cell align={left},
legend style={fill opacity=0.8, draw opacity=1, text opacity=1, at={(0.03,0.97)}, anchor=north west, draw=white!80!black},
log basis y={10},
tick align=outside,
tick pos=left,
x grid style={white!69.0196078431373!black},
xlabel={\(\displaystyle q\)},
xmin=1, xmax=13,
xtick style={color=black},
y grid style={white!69.0196078431373!black},
ymin=0.00338412039558712, ymax=2.94012120323038,
ymode=log,
ytick style={color=black}
]
\path [draw=black, semithick, dash pattern=on 5.55pt off 2.4pt]
(axis cs:0,1)
--(axis cs:13,1);

\addplot [semithick, color0]
table {%
1 0.00460291602902529
2 0.147697010730219
3 1.12720285060352
4 2.12595616461613
5 1.97778048030076
7 1.97311044152884
8 1.86854288246104
10 1.92868110564515
11 1.9506309673903
12 2.1616132179272
13 1.97822479825347
};
\addlegendentry{$\epsilon_{exp.}/\epsilon_{th}$}
\end{axis}

\end{tikzpicture}

%% file: SAGA_cst_est_C.tex
\begin{tikzpicture}

\definecolor{color0}{rgb}{0.12156862745098,0.466666666666667,0.705882352941177}

\begin{axis}[
legend cell align={left},
legend style={fill opacity=0.8, draw opacity=1, text opacity=1, draw=white!80!black},
tick align=outside,
tick pos=left,
x grid style={white!69.0196078431373!black},
xlabel={\(\displaystyle q\)},
xmin=1, xmax=13,
xtick style={color=black},
y grid style={white!69.0196078431373!black},
ymin=2.87269927106166, ymax=3.68396902596307,
ytick style={color=black}
]
\addplot [semithick, color0]
table {%
1 3.647093128013
2 2.98788352845044
3 2.96384405977239
4 2.92029746816107
5 2.99803371653575
7 2.95403540861517
8 2.90957516901173
10 2.960883328647
11 2.96613396580826
12 2.94021416683429
13 2.96806382449473
};
\addlegendentry{$\ea$}
\end{axis}

\end{tikzpicture}